\documentclass[11pt,reqno]{amsart}

\usepackage{amssymb}
\usepackage{amsmath}
\usepackage{amsthm}
\usepackage{amsfonts}
\usepackage{bbm}
\usepackage{enumitem}

\usepackage{subcaption}

\usepackage{a4wide}
\usepackage{bookmark}
\usepackage{hyperref}
\hypersetup{pdfstartview={FitH}}

\newtheorem{theorem}{Theorem}
\newtheorem{lemma}[theorem]{Lemma}

\newtheorem{proposition}[theorem]{Proposition}

\theoremstyle{definition}
\newtheorem{definition}{Definition}
\newtheorem{example}[definition]{Example}

\usepackage{tikz}
\usetikzlibrary{shapes,snakes,patterns}
\pgfdeclarelayer{background}
\pgfsetlayers{background,main}
\tikzstyle{vertex} = [fill,shape=circle,node distance=80pt]
\tikzstyle{chosen} = [draw, ultra thick,circle,node distance=80pt]
\tikzstyle{edge} = [fill,opacity=.5,fill opacity=.5,line cap=round, line join=round, line width=50pt]
\tikzstyle{elabel} =  [fill,shape=circle,node distance=30pt]

\pgfdeclarepatternformonly{soft crosshatch}{\pgfqpoint{-1pt}{-1pt}}{\pgfqpoint{4pt}{4pt}}{\pgfqpoint{3pt}{3pt}}%
{
  \pgfsetstrokeopacity{0.3}
  \pgfsetlinewidth{0.4pt}
  \pgfpathmoveto{\pgfqpoint{3.1pt}{0pt}}
  \pgfpathlineto{\pgfqpoint{0pt}{3.1pt}}
  \pgfpathmoveto{\pgfqpoint{0pt}{0pt}}
  \pgfpathlineto{\pgfqpoint{3.1pt}{3.1pt}}
  \pgfusepath{stroke}
}

\numberwithin{equation}{section}

\begin{document}

\title{T(1) theorem for dyadic singular integral forms associated with hypergraphs}


\author[M. Stip\v{c}i\'{c}]{Mario Stip\v{c}i\'{c}}
\address{Mario Stip\v{c}i\'{c}, Department of Mathematics, Faculty of Science, University of Zagreb, Bijeni\v{c}ka cesta 30, 10000 Zagreb, Croatia}
\email{mstipcic@math.hr}

\date{\today}



\begin{abstract}
This paper studies dyadic singular integral forms associated with $r$-partite $r$-uniform hypergraphs such that all their connected components are complete. We characterize their $\textup{L}^p$ boundedness by $T(1)$-type conditions in two different ways. We also dominate these forms by positive sparse forms and prove weighted estimates with multilinear Muckenhoupt weights.
\end{abstract}

\maketitle


\section{Introduction} \label{prvaprava}

Entangled multilinear singular integral forms have been studied by several authors over the last ten years; see the papers by Kova\v{c} \cite{K11}, \cite{K12}, Kova\v{c} and Thiele \cite{KT13}, Durcik \cite{D15}, \cite{D17}, and Durcik and Thiele \cite{DT18}. They recently found applications in ergodic theory \cite{K16}, \cite{DKST19}, in arithmetic combinatorics \cite{DK18}, \cite{DKR18}, to stochastic integration \cite{KS15}, and within the harmonic analysis itself \cite{DKT16}, \cite{DR18}. Therefore, it would be useful to have a reasonably general theory establishing (or characterizing) $\textnormal{L}^p$ bounds for these objects. As a step in this program we take results from the papers \cite{K11} and \cite{KT13}, where the forms are dyadic and indexed by bipartite graphs, and generalize them to $r$-partite $r$-uniform hypergraphs. Some higher-dimensional instances of dyadic entangled forms were already discussed by Kova\v{c} \cite{K12} and Durcik \cite{D14}, but our hypergraph generalization prefers a combinatorial description of the structure over a geometric one. Consequently, we can study less symmetric entangled forms and show their estimates in an open range of $\textnormal{L}^p$ spaces.

Working in a dyadic model certainly limits the applicability of our results, but this choice is justified in several ways. First, quite often dyadic models help in developing the techniques that are used later to approach the original, continuous-type problems. The reader can compare the present paper with the work of Durcik and Thiele \cite{DT18}, which is the current state-of-the-art on the continuous singular entangled forms. Second, in some applications it is possible to transfer an estimate easily from dyadic to continuous setting; see \cite{K12} and \cite{DR18}. Third, below we formulate an entangled T(1) theorem for dyadic forms associated with hypergraphs. Even its particular case dealing with graphs, which was discussed in \cite{KT13}, has not yet been formulated in the continuous setting and leaves an interesting open problem.

\begin{definition} A \emph{hypergraph} is an ordered pair $(V,E)$, where $V$ is a set of elements which we call \emph{vertices} and $E$ is a collection of nonempty subsets of $V\!$; the elements of $E$ are called \emph{edges}. Let $r \in \mathbb{N}$. A hypergraph $(V,E)$ is called \emph{$r$-partite} if there exists a partition of $V$ into $r$ nonempty parts $\left( V^{(i)} \right)_{1 \leq i \leq r}$ such that one cannot find $i \in \left\{ 1, 2, \dots, r \right\}$ and vertices $x,y \in V^{(i)}, x \neq y$ for which there would exist $e \in E$ such that $x, y \in e$. A hypergraph $(V,E)$ is called \emph{$r$-uniform} if each edge $e \in E$ has the cardinality $|e|=r$.
\end{definition}

Notice that every edge $e$ of an $r$-partite $r$-uniform graph $(V,E)$ with an associated $r$-partition of the vertex set $V = \bigcup_{i=1}^r V^{(i)}$ satisfies $\left| e \cap V^{(i)} \right| = 1$ for every $i \in \left\{ 1, 2, \dots, r \right\}$. In other words, each edge contains exactly one vertex from each of the vertex-partition parts. In this situation each edge can be identified with an element of $V^{(1)} \times V^{(2)} \times \dots \times V^{(r)} = \prod_{i=1}^r V^{(i)}$.

Each hypergraph $(V,E)$ can be partitioned into connected components, i.e.\@ there exist partitions $\left( V_j \right)_{1 \leq j \leq k}$ of $V$ and $\left( E_j \right)_{1 \leq j \leq k}$ of $E$ such that each subhypergraph $\left( V_l, E_l \right)$, $l \in \left\{ 1, 2, \dots, k \right\}$ is connected (i.e.\@ for each $x,y \in V_l$ there exist $n \in \mathbb{N}$, $v_1, \dots, v_{n-1} \in V_l$ and $e_1, \dots, e_n \in E_j$ such that $x,v_1 \in e_1, v_1,v_2 \in e_2, \dots, v_{n-1},y \in e_n$) and maximal (i.e.\@ it is not contained in any other connected subgraph of $(V,E)$). For each such $l$ and for each $i \in \{ 1, \dots, r \}$ we define $V^{(i)}_l := V_l \cap V^{(i)}.$ This makes $(V^{(i)}_l)_{i=1}^r$ an $r$-partition of the set $V_l$, which goes along with the hypergraph $H_l$ being $r$-partite. For each $e \in E$, taking the unique $l \in \left\{ 1, \dots, k \right\}$ such that $e \in E_l$, we define
\begin{equation}d_e := \max\limits_{1 \leq i \leq r} \prod\limits_{\substack{1 \leq j \leq r \\ j \neq i}} | V^{(j)}_l |. \label{thresholds}\end{equation}
In words, $d_e$ is the product of cardinalities of the $r-1$ largest vertex-partition parts of the connected component containing $e$. These quantities will turn out to be important in determining the ranges of exponents of the estimates to follow.

We will say that an $r$-partite $r$-uniform hypergraph $H = (V,E)$ is \emph{a complete hypergraph} if $E = \prod_{i=1}^r V^{(i)}$. In this article we are only considering the forms represented by the hypergraphs such that all of their connected components are complete hypergraphs. Denote $n_i := |V^{(i)}|$ and $n := |V| = \sum_{i=1}^{r}n_i$. We are going to work with the assumption $\min_{1 \leq i \leq r} n_i \geq 2,$ from which we easily deduce that
\begin{equation}\sum_{e \in E} \frac{1}{d_e} > 1 \label{eq:condition}\end{equation}
is valid. The main theorem below would not give any estimates for forms associated with hypergraphs if (\ref{eq:condition}) failed.

For $r \in \mathbb{N}$, we define
$$\mathcal{C}_r := \bigg\{ \prod_{i=1}^r \big[ 2^kl_i, 2^k\left( l_i+1 \right) \big> : k, l_i \in \mathbb{Z}, i=1,2,\dots,r \bigg\},$$
the set of dyadic cubes in $\mathbb{R}^r$. Also, for nonnegative quantities $A$ and $B$ we write $A\lesssim B$ if $A\leq C B$ holds with some unimportant finite constant $C$.

The following setting is a higher-dimensional multilinear generalization of the dyadic setup from the paper \cite{AHMTT02} by Auscher, Hofmann, Muscalu, Tao, and Thiele. Let $K: \mathbb{R}^n \rightarrow \mathbb{C}$ be \emph{a perfect dyadic Calder\' on-Zygmund kernel}, i.e.\@ a locally integrable, bounded and compactly supported function that is constant on each $n$-dimensional dyadic cube not intersecting the diagonal
\[ D = \big\{ ( \underbrace{x^{(1)}, \dots, x^{(1)}}_{n_1 \,\, \textnormal{times}}, \dots, \underbrace{x^{(r)}, \dots, x^{(r)}}_{n_r \,\, \textnormal{times}}) \in \mathbb{R}^n \big\} \]
and that, for each $\mathbbm{x} = (x_1^{(1)}, \dots, x_{n_1}^{(1)}, \dots, x_1^{(r)}, \dots, x_{n_r}^{(r)}) \in \mathbb{R}^n \backslash D$, satisfies
\begin{equation}| K(\mathbbm{x}) | \lesssim \bigg( \sum_{i=1}^r \sum_{1 \leq j_1 < j_2 \leq n_i} |x_{j_1}^{(i)}-x_{j_2}^{(i)}| \bigg)^{r-n}. \label{CalZyg}\end{equation}
For a tuple $\mathbf{F} = (F_e)_{e\in E}$  of measurable bounded functions we define
$$\Lambda_E \left( \mathbf{F} \right) := \int_{\mathbb{R}^n} \bigg( \prod_{e \in E} F_e ( \mathbbm{x}_e ) \bigg) K(\mathbbm{x}) d\mathbbm{x},$$
and, for fixed $e_0 \in E$ and for each $\mathbbm{x}_{e_0} \in \mathbb{R}^r$,
\begin{equation}T_{e_0} \left( \mathbf{F}_{E \backslash \left\{ e_0 \right\}} \right) (\mathbbm{x}_{e_0}) := \int_{\mathbb{R}^{n-r}} \bigg( \prod_{e \in E \backslash \left\{ e_0 \right\}} F_e (\mathbbm{x}_e) \bigg) K(\mathbbm{x}) \prod_{v \in V \backslash e_0} dx_v, \label{pomt}\end{equation}
where we have denoted the tuple $(F_e)_{e\in E\setminus\{e_0\}}$ simply by $\mathbf{F}_{E\setminus\{e_0\}}$. We can notice that we have
$$\Lambda_E \left( \mathbf{F} \right) = \int_{\mathbb{R}^r} T_{e_0} \left( \mathbf{F}_{E \backslash \left\{ e_0 \right\}} \right) (\mathbbm{x}_{e_0}) F_{e_0}(\mathbbm{x}_{e_0}) d\mathbbm{x}_{e_0}$$
for each $e_0 \in E$. For the statement of the main theorem, we are also going to need \emph{a dyadic BMO-seminorm}, which we define as
\begin{align*}\left\| F \right\|_{\textnormal{BMO}(\mathbb{R}^r)} := \sup_{Q\in\mathcal{C}_r} \left( \frac{1}{|Q|} \int_Q \left| F - \frac{1}{|Q|} \int_Q F \right|^2 \right)^{\frac{1}{2}}. \end{align*}

Even though we are primarily interested in the $\textup{L}^p$ estimates for the above multilinear forms, we prefer to give the arguments that also yield sparse domination. The notion of sparse collections of cubes and the associated sparse forms was introduced by Lerner \cite{L13}; the reader can also compare the dyadic setting of Lerner and Nazarov \cite{LN15}. Since we are dealing with multilinear forms, we will need the multilinear modification of the theory developed by Culiuc, Di Plinio, and Ou \cite{CPO18}, so several major concepts and many ideas of proofs will be adapted from that paper.

\begin{definition} \label{defsparse} For a fixed $c>0$ we say that $\mathcal{S} \subseteq \mathcal{C}_r$ is \emph{a sparse family} if it is a collection of dyadic cubes such that, for each $Q \in \mathcal{S}$, there exists a measurable set $E_Q \subseteq Q$ with the following properties:
\begin{itemize} \item for each $Q \in \mathcal{S}$ we have $| E_Q | \geq c|Q|$,
\item for each $Q,Q' \in \mathcal{S}, Q \neq Q'$, sets $E_Q$ and $E_{Q'}$ are mutually disjoint.
\end{itemize}
\emph{A sparse (multisublinear) form} associated with $\mathcal{S}$ is given by
$$\Theta_{\mathcal{S}}( \mathbf{F} ) := \sum_{Q \in \mathcal{S}} |Q| \prod_{e \in E} [ |F_e|^{d_e} ]_Q^{\frac{1}{d_e}}.$$
\end{definition}

As is well-known, the sparse domination also implies weighted estimated for the form in question. Once again, we merely adapt the setting from the paper \cite{CPO18} by Culiuc, Di Plinio, and Ou. Given the edge-set $E$ with its collection of integers $\textbf{d}=(d_e)_{e\in E}$ defined as before, let $\mathbf{p}=(p_e)_{e \in E}$ be an arbitrary tuple of exponents from $[1,\infty]$ such that $p_e>d_e$ for each $e\in E$ and $\sum_{e\in E}\frac{1}{p_e}=1$. Also, let $\mathbf{w} = (w_e)_{e \in E}$ be a tuple of strictly positive functions satisfying
\begin{equation} \prod_{e \in E} w_e^{\frac{1}{p_e}} = 1. \label{condMuck} \end{equation}
We will define \emph{the multilinear Muckenhoupt constant} of the tuple $\mathbf{w}$ to be an expression
$$[\mathbf{w}]_{\mathbf{p},\mathbf{d}} := \sup_{Q \in \mathcal{C}_r} \prod_{e \in E} \big[ w_e^{\frac{-d_e}{p_e-d_e}} \big]_Q^{\frac{1}{d_e}-\frac{1}{p_e}}.$$
In the following result we will consider \emph{the weighted $\textnormal{L}^p$} space along with \emph{the weight $w$}, denoted as $\textnormal{L}^p(w)$ and defined as the standard $\textnormal{L}^p$ space according to the measure $\nu$ such that $d\nu = w d\lambda$, $w$ being a nonnegative function and $\lambda$ being the standard Lebesgue measure.

We are ready to state the main result of this paper.

\begin{theorem}\label{T1thm} The following statements are equivalent. \begin{enumerate}[label=(\alph*)] 
\item The weak boundedness property
\begin{align}| \Lambda_E (\mathbbm{1}_Q, \dots, \mathbbm{1}_Q) | \lesssim |Q| \,\,\, \textnormal{for each} \,\, Q \in \mathcal{C}_r \label{first}\end{align}
and the T(1)-type conditions
\begin{align}\left\| T_e (\mathbbm{1}_{\mathbb{R}^r}, \dots, \mathbbm{1}_{\mathbb{R}^r}) \right\|_{\textnormal{BMO}(\mathbb{R}^r)} \lesssim 1 \,\,\, \textnormal{for each} \,\,\,  e \in E \label{second}\end{align}
are valid.

\item We have
\begin{equation}\| T_{e_0}( \mathbbm{1}_Q )_{e \in E \backslash \{ e_0 \}} \|_{\textnormal{L}^1(Q)} \lesssim |Q| \,\, \textnormal{for each} \,\, e_0 \in E \,\, \textnormal{and} \,\, Q \in \mathcal{C}_r. \label{assum}\end{equation}

\item The form $\Lambda_E$ satisfies the estimate
\begin{equation}| \Lambda_E \left( \mathbf{F} \right) | \lesssim \prod_{e \in E} \left\| F_e \right\|_{\textnormal{L}^{p_e}(\mathbb{R}^r)} \label{T1bound}\end{equation}
for all choices of exponents $d_e < p_e \leq \infty, e \in E,$ such that $\sum_{e \in E} \frac{1}{p_e} = 1$.

\item The form $\Lambda_E$ satisfies the estimate (\ref{T1bound}) for some choice of exponents $d_e < p_e \leq \infty, e \in E,$ such that $\sum_{e \in E} \frac{1}{p_e} = 1$.

\item For any measurable, bounded, and and compactly supported tuple of functions $\mathbf{F}$ there exists a sparse form $\Theta_{\mathcal{S}}$ for which we have $| \Lambda_E( \mathbf{F} ) | \lesssim \Theta_{\mathcal{S}}( \mathbf{F} )$.

\item Let $\mathbf{p}=(p_e)_{e \in E}$ be an arbitrary tuple of exponents from $[1,\infty]$ such that $p_e>d_e$ for each $e\in E$ and $\sum_{e\in E}\frac{1}{p_e}=1$ and $\mathbf{w} = (w_e)_{e \in E}$ a tuple of strictly positive functions satisfying (\ref{condMuck}). For each tuple $\mathbf{F} = (F_e)_{e \in E}$ we have
$$| \Lambda_E( \mathbf{F} ) | \lesssim [ \mathbf{w} ]_{\mathbf{p},\mathbf{d}}^{\max_{e \in E}\frac{p_e}{p_e-d_e}} \prod_{e \in E} \| F_e \|_{\textnormal{L}^{p_e}(w_e)}.$$
\end{enumerate}

The implicit constants in all of the above estimates depend on the hypergraph $H$, kernel $K$, the exponents in question, and they also mutually depend on each other.
\end{theorem}

As we have already mentioned, the range of the exponents $\mathbf{p}$ appearing in parts (c) and (f) is nonempty because of \eqref{eq:condition}. In the particular case dealing with bipartite graphs (without the completeness assumption), i.e.\@ when $r=2$, the paper \cite{KT13} proceeds by studying exceptional cases, so that all nondegenerate bipartite graphs are covered with some nonempty range of exponents. We are not able to do the same here, since higher dimensions bring an additional structural complexity, and this is another reason why we find convenient to assume that each hypergraph component is complete. Indeed, the reader can see the recent paper by Durcik and Roos \cite{DR18} for an example of an open problem in dimensions $r\geq 4$, which would be resolved if we could apply our main result to the hypergraph in question.

\begin{example}Let us illustrate how the twisted paraproduct form from the paper \cite{K12} can be represented as an entangled form associated to a hypergraph.

Suppose that each of the partition classes $V^{(i)}$ has precisely two vertices and suppose that the hypergraph is complete, so that indeed $E = \prod_{i=1}^r V^{(i)},$ $|E|=2^r$. Thus, the set of edges is in a bijective correspondence with $\{ 0,1 \}^r$ and we are working with a tuple of functions $\mathbf{F} = (F_{j_1,j_2,\dots,j_r})_{j_1,j_2,\dots,j_r \in \{ 0,1 \}}$. For the kernel we take
$$K(\mathbbm{x}) := \sum_{Q = \prod_{i=1}^r I^{(i)}} |I^{(1)}|^r \mathbbm{h}_{I^{(1)}}^1(x_1^{(1)})\mathbbm{h}_{I^{(1)}}^1(x_2^{(1)}) \prod_{k=2}^r \mathbbm{h}_{I^{(k)}}^0(x_1^{(k)})\mathbbm{h}_{I^{(k)}}^0(x_2^{(k)}),$$
where the summation is performed over all dyadic cubes contained in $[-2^N,2^N\rangle^r$ with edge-length at least $2^{-N}$, for some positive integer $N$. We also remark that $\mathbbm{h}_{I^{(1)}}^0$ and $\mathbbm{h}_{I^{(1)}}^1$ are the non-cancellative and cancellative Haar functions defined at the very beginning of the next section. It is easy to verify that $K$ and the associated form $\Lambda_E$ satisfy conditions from part (a) of the T(1) theorem with
$$T_e(\mathbbm{1}_{\mathbb{R}^r},\dots,\mathbbm{1}_{\mathbb{R}^r}) = 0$$
for each $e \in E$. Consequently, we obtain $\textnormal{L}^p$ estimates for $\Lambda_E$ in the range $2^{r-1} < p_e \leq \infty$ for each $e \in E$, $\sum_{e \in E} \frac{1}{p_e} = 1.$

The most interesting case in \cite{K12} is obtained by taking $F_{j_1,j_2,\dots,j_r} = \mathbbm{1}_{\mathbb{R}^r}$ whenever $j_1+j_2+\dots+j_r \geq 2$, which leaves us with only $r+1$ nontrivial functions. For the remaining functions we need to take $p_e = \infty$, which makes the range of exponents empty unless $(r+1)\frac{1}{2^{r-1}} > 1$, i.e.\@ unless $r \leq 2$. The case $r=3$ was handled in \cite{DR18}, while the cases $r \geq 4$ are still open at the time of writing.
\end{example}

The paper is organized as follows. In Section \ref{drugaprava} we give additional definitions, which will be helpful for writing down the proofs. Section \ref{trecaprava} gives the boundedness of the form localized on the finite convex tree. This result serves as the key idea for most of the proofs that follow. In Section \ref{cetvrtaprava} one can find the results on the boundedness of localized cancellative and non-cancellative entangled paraproducts, which are obtained by performing the so-called cone decomposition of the kernel. Finally, the proof of the main theorem is completed in Section \ref{petaprava}.

\section{Notation and terminology} \label{drugaprava}

The elements of $\mathcal{C}_r$ will usually be denoted as $I_1 \times I_2 \times \dots \times I_r = \prod_{i=1}^r I_i$, with $I_1, \dots, I_r \in \mathcal{C}_1$. For $I \in \mathcal{C}_1$, let
$$\mathbbm{h}^0_I := \frac{1}{\left| I \right|} \mathbbm{1}_I, \,\,\,\,\, \mathbbm{h}^1_I := \frac{1}{\left| I \right|} \left( \mathbbm{1}_{I_L}-\mathbbm{1}_{I_R} \right),$$
where $I_L$ and $I_R$ are, in order, left and right halves of the interval $I$; more precisely, if $I = \left[ a,b \right>$ for some $a,b \in \mathbb{R}$, then $I_L := \left[ a, \frac{a+b}{2} \right>$ and $I_R := \left[ \frac{a+b}{2}, b \right>$. Function $h_I^0$ is simply the $\textnormal{L}^1$-normalized characteristic function of $I$, while $h_I^1$ is the so-called \emph{Haar function} normalized in the $\textnormal{L}^1$ sense. We will also call these, in order, \emph{non-cancellative} and \emph{cancellative} Haar functions.

Let $(V,E)$ be an $r$-partite $r$-uniform hypergraph with a fixed $r$-partition; denote $V^{(i)} = \{ v^{(i)}_1, \dots, v^{(i)}_{n_i} \}$ for each $i \in \{ 1, \dots, r \}$. Let $\left( F_e \right)_{e \in E}$ be a tuple of measurable, bounded and compactly supported functions from $\mathbb{R}^r$ to $\mathbb{R}$ and take $S = (S^{(i)})_{1 \leq i \leq r}$ with $S^{(i)} \subseteq V^{(i)}$ for each $i \in \{ 1, \dots, r \}$ and such that there exists $i_0 \in \{ 1, \dots, r \}$ such that $|S^{i_0}| \geq 2$. We define
\begin{align*}&\Lambda_{E, S} \left( \mathbf{F} \right) := \!\!\!\sum_{Q = I_1 \times \dots \times I_r \in \mathcal{C}_r}\!\!\! \left| Q \right| \int_{\mathbb{R}^n} \Big( \prod_{e \in E} F_e\left( \mathbbm{x}_e \right) \Big) \prod_{i=1}^r \Big( \prod_{v^{(i)} \in S^{(i)}} \mathbbm{h}^1_{I_i}( x_{v^{(i)}} ) \prod_{v^{(i)} \in ( S^{(i)} )^c} \mathbbm{h}^0_{I_i} ( x_{v^{(i)}} ) \Big) d\mathbbm{x},\end{align*}
where
\[\mathbbm{x} = \big( x_{v^{(1)}_1}, \dots, x_{v^{(1)}_{n_1}}, \dots, x_{v^{(r)}_1}, \dots, x_{v^{(r)}_{n_r}} \big),\]
$n=|V|$ and, for $e = (v^{(1)}, \dots, v^{(r)}) \in E$,
$$\mathbbm{x}_e := \big( x_{v^{(1)}}, \dots, x_{v^{(r)}} \big).$$
 
For the definition of the form $\Lambda$ and for the statement of the main problem we intentionally labeled functions $F_e$ with the set of edges $E$ and variables $x_{v^{(i)}}$ with the set of vertices $V$ and its $r$-partition $(V^{(i)})_{1 \leq i \leq r}$. As we will see later, we will introduce a short, compact notation which encodes all important information by just defining a certain labeled ($r$-partite and $r$-uniform) hypergraph, therefore making proofs easier to write and more practical to visualise. With this, a tuple $S$ of vertices will be considered as a tuple of \emph{selected vertices}.

\begin{definition} A \emph{labeled hypergraph} is any hypergraph $(V,E)$ along with sets $L_V$ and $L_E$, an injective function $l_V : V \rightarrow L_V$ and an arbitrary function $l_E : E \rightarrow L_E$. The elements of sets $L_V$ and $L_E$ will be called, in order, \emph{vertex labels} and \emph{edge labels}. Note that vertex labels are required to be different, but we allow the repetition of edge labels.\end{definition}

Given an $r$-partite hypergraph $(V,E)$ and the corresponding partition of $V$ as $\left( V^{(i)} \right)_{1 \leq i \leq r}$, we will usually denote vertices as $V^{(i)} = \big\{ v_1^{(i)}, v_2^{(i)}, \dots, v_{n_i}^{(i)} \big\}$ for each $i \in \left\{ 1, 2, \dots, r \right\}$. Similarly, we will write $L_V := \cup_{i=1}^r L^{(i)}_V$ and $L^{(i)}_V := \{ x^{(i)}_j : j \in \mathbb{N} \}$ for each $i \in \{ 1, \dots, r \}$; with this notation, we will assume that $l_V(V^{(i)}) \subseteq L^{(i)}_V$ for each $i \in \{ 1, \dots, r \}$. For shorter notation we may write $x_v := l_V(v)$ for each $v \in V$. The elements of sets $L_V$ and $L_E$ will be substituted with real variables and real-valued functions.

Take a tuple $\mathbf{F}=(F_l)_{l\in L_E}$ of nonnegative measurable compactly supported functions. More precisely, this is a collection of functions indexed by the set $L_E$ and these functions will be substituted in the places of edge labels in all of the following analytical expressions. For $Q=\prod_{i=1}^r I_i \in\mathcal{C}_r$ an \emph{evaluation of a tuple $\mathbf{F}$ on the hypergraph $H$}, given $S$ and $Q$ is defined as the number given by
$$\left[ \mathbf{F} \right]_{H,S,Q} := \int_{\mathbb{R}^n} \prod_{e \in E} F_{l_E(e)} (\mathbbm{x}_e) \prod_{i=1}^r \Big( \prod_{v^{(i)} \in S^{(i)}} \mathbbm{h}^1_{I_i} \big( x_{v^{(i)}} \big) \prod_{v^{(i)}\in ( S^{(i)} )^c} \mathbbm{h}^0_{I_i} \big( x_{v^{(i)}} \big) \Big) d\mathbbm{x},$$
where $S = (S^{(i)})_{1 \leq i \leq r}$ and $S^{(i)} \subseteq V^{(i)}$ for each $i \in \left\{ 1, 2, \dots, r \right\}$. This expression will also be called \emph{paraproduct-type term}. In particular, if each $S^{(i)} = \emptyset$, then the mapping $\mathcal{A}\colon \mathbf{F} \mapsto [\mathbf{F}]_{H,S,Q}$ will be called an \emph{averaging paraproduct-type term}. Also, any linear combination of paraproduct-type terms will be called a \emph{paraproduct-type expression}. Note that the form $\Lambda_{E,S}$ that we are trying to bound has much more compact notation now:
$$\Lambda_{E,S} \left( \mathbf{F} \right) = \sum_{Q \in \mathcal{C}_r} |Q| \left[ \mathbf{F} \right]_{H,S,Q},$$
where the initial labelling of edges is given by $l_E(e):=F_e$.

For a complex-valued locally integrable function $F$ and any dyadic cube $Q$ we also introduce the notation
\[ [F]_Q := \frac{1}{|Q|}\int_Q F(x) dx. \]
In words, $[F]_Q$ is simply the average of $F$ on $Q$.

Throughout the whole paper we are going to work with the functions $F_e, e \in E$ that are nonnegative, as the general result will follow by representing each of these functions as a difference of its positive and negative parts.

\section{Boundedness of dyadic singular integral forms associated with hypergraphs} \label{trecaprava}

For each $r \in \mathbb{N}$ and $Q \in \mathcal{C}_r$ there exist exactly $2^r$ disjoint cubes $Q_1, \dots, Q_{2^r} \in \mathcal{C}_r$ such that $\left| Q_1 \right| = \dots = \left| Q_r \right| = 2^{-r}|Q|$ and $Q_1, \dots, Q_{2^r} \subseteq Q$. These $Q_i, i \in \left\{ 1, \dots, 2^r \right\}$ are called the \textit{children} of $Q$, while $Q$ is the \textit{parent} of $Q_1, \dots, Q_{2^r}$. The family of children of a cube $Q$ will be denoted as $\mathcal{C}(Q)$.

\begin{definition} Let $r \in \mathbb{N}$. A \emph{tree} is a family $\mathcal{T} \subseteq \mathcal{C}_r$ for which there exists $Q_{\mathcal{T}} \in \mathcal{T}$ such that $Q \subseteq Q_{\mathcal{T}}$ for every $Q \in \mathcal{T}$; such $Q_{\mathcal{T}}$ is called a \emph{root} of the tree $\mathcal{T}$. A tree $\mathcal{T}$ is called \emph{convex} if for every $Q_1,Q_3 \in \mathcal{T}$ and $Q_2 \in \mathcal{C}_r$ the inclusion $Q_1 \subseteq Q_2 \subseteq Q_3$ implies $Q_2 \in \mathcal{T}$. A \emph{leaf} of the tree $\mathcal{T}$ is any $Q \in \mathcal{C}_r \backslash \mathcal{T}$ with the parent $Q' \in \mathcal{T}$. A family of these cubes will be marked as $\mathcal{L} \left( \mathcal{T} \right)$.
\end{definition}

Given $Q \in \mathcal{C}_r$, for an expression $\mathcal{B} = \mathcal{B}_Q \left( \mathbf{F} \right)$ we define
$$\square \mathcal{B}_Q \left( \mathbf{F} \right) := \sum_{Q' \in \mathcal{C}(Q)} \frac{1}{2^r} \mathcal{B}_{Q'}\left( \mathbf{F} \right) - \mathcal{B}_Q \left( \mathbf{F} \right).$$
The operator $\Box$ can be thought of as a certain discrete version of the Laplace operator.

\begin{proposition} \label{difference} For any cube $Q$, the first-order difference of the averaging paraproduct-type term $\mathcal{B}_{H,Q} = \left[ \mathbf{F} \right]_{H,(\emptyset),Q}$ is the paraproduct-type expression
$$\square \mathcal{B}_{H,Q} = \!\!\!\!\!\!\!\!\!\!\!\!\!\!\!\!\!\!\!\!\sum_{\substack{\left( \forall i \in \left\{ 1, \dots, r \right\} \right) S^{(i)} \subseteq V^{(i)}, \left| S^{(i)} \right| \, \mathrm{even} \\ \left( \exists i_0 \in \left\{ 1, \dots, r \right\} \right) \left| S^{(i_0)} \right| \neq 0}}\!\!\!\!\!\!\!\!\!\!\!\!\!\!\!\!\!\!\!\! \left[ \mathbf{F} \right]_{H,S,Q}.$$
\end{proposition}

\begin{proof} The proof of this result is similar to the proof of Theorem $2.1$ from \cite{K11}, where one can look for a more detailed proof. The idea is to start with the identity
$$\prod_{i=1}^r \Big( \prod_{j=1}^{n_i} \big( 1+\alpha_i \nu_{I_i} \big( x_{v^{(i)}_j} \big) \big) \Big) = \sum_{\substack{S^{(1)} \subseteq V^{(1)} \\ \dots \\ S^{(r)} \subseteq V^{(r)}}} \prod_{i=1}^r \prod_{v^{(i)}_j \in S^{(i)}} \alpha_i^{\left| S_i \right|} \nu_{I_i} \big( x_{v^{(i)}_j} \big)$$
for dyadic intervals $I_1, \dots, I_r$ of same length and with $\nu_I := \mathbbm{1}_{I_L}-\mathbbm{1}_{I_R}$ for each $I \in \mathcal{C}_1$. If we sum these equations based on the choice of $\alpha_i = \pm 1, 1 \leq i \leq r$, multiply the new equation with $\frac{\prod_{e \in E} F_{l_E(e)} (\mathbbm{x}_e)}{2^r \left|I_1 \right|^{n_1} \dots  \left|I_r \right|^{n_r}}$, integrate over all the appearing variables and use the identities
$$\mathbbm{h}^1_I = \nu_I \mathbbm{h}^0_I,\,\,\,\,\,\, \mathbbm{h}^0_{I_L} = \left( 1 + \nu_I \right) \mathbbm{h}^0_I,\,\,\,\,\,\, \mathbbm{h}^0_{I_R} = \left( 1 - \nu_I \right) \mathbbm{h}^0_I,$$
we get the equation that can also be written as
$$\frac{1}{2^r} \sum_{Q' \in \mathcal{C}(Q)}\left[ \mathbf{F} \right]_{H,(\emptyset),Q'} = \!\!\!\!\!\!\!\!\!\!\sum_{\substack{S^{(1)} \subseteq V^{(1)}, \left| S^{(1)} \right| \, \mathrm{even} \\ \dots \\ S^{(r)} \subseteq V^{(r)}, \left| S^{(r)} \right| \, \mathrm{even}}}\!\!\!\!\!\!\!\!\!\! \left[ \mathbf{F} \right]_{H,S,Q}.$$
By substracting $\left[ \mathbf{F} \right]_{H,(\emptyset),Q}$ from both sides of the equality we get the desired result.
\end{proof}

Another useful result is the following lemma from \cite{K11}, stated in a general notation instead of a notation using hypergraphs and selected vertices.

\begin{lemma} \label{neneg} For any $m \in \mathbb{N}$, any $I_1, \dots, I_m \in \mathcal{C}_1$ and any nonnegative function $f : \mathbb{R}^m \rightarrow \mathbb{R}$, the expression
$$\sum_{S \subseteq \left\{ 1, \dots, m \right\}, \left| S \right| \, \mathrm{even}} \int_{\mathbb{R}^n} f(x_1, \dots, x_m) \bigg( \prod_{i \in S} \mathbbm{h}_{I_i}^1(x_i) \prod_{i \in S^c} \mathbbm{h}_{I_i}^0(x_i) \bigg) d(x_1,\dots,x_n)$$
is also nonnegative.
\end{lemma}

The next lemma is a straightforward repeated application of H\"{o}lder's inequality. The reader can compare it with the particular case $r=2$ appearing in \cite{K11}. 

\begin{lemma} \label{holder} Let $H=(V,E)$ be a complete $r$-partite $r$-uniform labeled hypergraph. If $N := \prod_{i=1}^r n_i$ and $M := \max\left\{ \frac{N}{n_1}, \dots, \frac{N}{n_r} \right\} = \max_{1 \leq i \leq r} \prod_{\substack{1 \leq j \leq r \\ j \neq i}} n_j$, then for any tuple $\mathbf{F}=(F_l)_{l\in L_E}$ of nonnegative measurable functions we have\vspace*{-0.1cm}
$$\big[ (F_{l_E(e)})_{e \in E} \big]_{H,(\emptyset), Q} \leq \prod_{e \in E} \big[ F_{l_E(e)}^M \big]_Q^{\frac{1}{M}}.$$\vspace*{-0.5cm}
\end{lemma}

Let $H=(V,E)$ be a labeled hypergraph with the label functions $l_V$ and $l_E$ and the set of vertex labels marked as $L_V = \cup_{i=1}^r L^{(i)}_V$ such that, for each $i \in \left\{ 1, \dots, r \right\}$, $l_V \left( V^{(i)} \right) \subseteq L^{(i)}_V$. With a slight deviation from the previous notation, this time we will write $L^{(i)}_V = \{ x^{(i)}_{j,k} : j,k \in \mathbb{N} \}$. In the following proofs we will square certain parts of paraproduct-type terms, making certain variables appear more than once. To keep the practical notation of the evaluation of the expression at certain graph, we are expanding the vertex label sets with ``copies", i.e.\@ as certain variable $x^{(i)}_j$ can appear more than once (but at most $n_i = |V^{(i)}|$ times), so we will mark its copies with $x^{(i)}_{j,k}, k \in \mathbb{N}$. It will also be practical to denote $L^{(i)}_j := \{ x^{(i)}_{j,k}: k \in \mathbb{N} \}$.

Let us introduce a requirement on $l_V$ and $l_E$ so that they produce ``properly'' labelled hypergraphs, i.e., those that appear in the following proof. Take any label variable $x_{i_k,j_k}^{(k)}$ for $k \in \left\{ 1, \dots, r \right\}, i_k, j_k \in \mathbb{N}$. As $l_V$ is an injective function, whenever $x_{i_k,j_k}^{(k)} \in \textnormal{Im}(l_V)$ we can define $v^{(k)}_{i_k,j_k} := l_V^{-1}(x_{i_k,j_k}^{(k)})$. For the set of edge labels we choose $L_E := \left\{ F_{i_1, \dots, i_r} : i_1, \dots, i_r \in \mathbb{N} \right\}$. With this notation, we will require the following condition to be satisfied:
\begin{equation}l_E( (v^{(1)}_{i_1,j_1}, \dots, v^{(r)}_{i_r,j_r}) ) = F_{i_1, \dots, i_r} \label{isti}\end{equation}
for each choice of indices $i_k$ and $j_k$, $k=1, \dots, r$; let us remember the agreement of the edges being the elements of the set $\prod_{i=1}^r V^{(i)}$, as mentioned in the Section \ref{prvaprava}. This means that any two edges with the same first lower indices of their vertices receive the same label from the set $L_E$. Otherwise, the two edges receive different labels. 

Additionally, we will restrict our attention to hypergraphs that are ``proper" in the sense that, when two variables $x_{v_1}$ and $x_{v_2}$ have the same first lower indices, then they belong to the same connected component and to the same partition class. Notice that this property is trivially satisfied in case of the complete hypergraph.

Without loss of generality, we will consider only those labeled hypergraphs for which the tuple $\big( \big| l_V^{-1} ( L^{(i)}_1) \big|, \big| l_V^{-1} (L^{(i)}_2)\big|, \dots, \big| l_V^{-1} ( L^{(i)}_{m_i} ) \big| \big)$ is decreasing for each $i \in \left\{ 1, \dots, r \right\}$; otherwise we could interchange the roles of the vertex labels (along with their copies) in a way that this becomes the decreasing tuple. That way we would operate with labeled hypergraphs with same set of vertices $V$ and same set of vertex labels $L_V$, but with a different label function $l_V$. A family of such hypergraphs on the set of vertices $V = \cup_{i=1}^r V^{(i)}$ will be denoted by $\mathcal{H}_{(n_i)}$. We define
$$\mathcal{S} := \big\{ S = \big( S^{(i)} \big)_{1 \leq i \leq r} : S \neq \left( \emptyset \right) \,\, \textnormal{and, for all} \,\, i \in \left\{ 1, 2, \dots, r \right\}, \,\, S^{(i)} \subseteq V^{(i)} \,\,\mathrm{and}\,\, | S^{(i)} | \,\,\textrm{is even} \big\}$$
Also, we define a binary relation $\preceq$ for hypergraphs $H, H' \in \mathcal{H}_{(n_i)}$ in the following way.
\begin{align*}
H \preceq H' \Longleftrightarrow& \big( \big| l_V^{-1} ( L^{(1)}_1) \big|, \dots, \big| l_V^{-1} ( L^{(1)}_{m_1} ) \big|; \dots;  \big| l_V^{-1} ( L^{(r)}_1) \big|, \dots, \big| l_V^{-1} ( L^{(r)}_{m_r} ) \big| \big) \\
&\geq \big( \big| l_V^{'-1} ( L^{(1)}_1) \big|, \dots, \big| l_V^{'-1} ( L^{(1)}_{m_1} ) \big|; \dots;  \big| l_V^{'-1} ( L^{(r)}_1) \big|, \dots, \big| l_V^{'-1} ( L^{(r)}_{m_r} ) \big| \big),
\end{align*}
where we consider the latter relation on $m$-tuples to be a standard lexicographical order. We can notice that $(\mathcal{H}_{(n_i)}, \preceq)$ is a totally ordered finite set; therefore there exist minimal and maximal hypergraphs with respect to this relation.

The first case we are going to cover in our proofs is when the hypergraph we are working with is complete. Consequently, the numbers $d_e$ are the same for all edges $e$ and we write them simply as $d$. Moreover, we fix a finite convex tree $\mathcal{T}$. Any constants in the inequalities will be independent of the choice of that tree.

Finally, let us also, for a moment, assume that all functions constituting $\mathbf{F}$ are normalized so that
\[ \max_{Q\in\mathcal{T}\cup\mathcal{L}(\mathcal{T})} [F_l^d]^{1/d}=1. \]
for each $l\in L_E$. Later we will use homogeneity to remove this normalizing condition.

\begin{lemma} \label{bound} For every complete $r$-partite $r$-regular hypergraph $H \in \mathcal{H}_{(n_i)}$ there exists an averaging paraproduct-type term $\mathcal{B}_H = \mathcal{B}_{H,Q}$ satisfying
$$\max_{Q \in \mathcal{T} \cup \mathcal{L}(\mathcal{T})} \mathcal{B}_{H,Q} \lesssim_{\left( n_i \right)} 1$$
and such that for every $\delta \in \left< 0, 1 \right>$ and for every $Q \in \mathcal{C}_r$ the following inequality holds for some $C_{(n_i)}>0$:
$$| \left[ \mathbf{F} \right]_{H,S,Q} | \leq \square \mathcal{B}_{H,Q} + C_{(n_i)} \delta^{-1} \!\!\!\!\!\!\!\!\!\!\sum_{\substack{H' \in \mathcal{H}_{(n_i)}, H' \prec H \\ R \in \mathcal{S}}}\!\!\!\!\!\!\!\!\!\! | \left[ \mathbf{F} \right]_{H',R,Q} | + C_{(n_i)} \delta \!\!\!\!\!\!\!\!\!\!\sum_{\substack{H' \in \mathcal{H}_{(n_i)}, H' \succeq H \\ R \in \mathcal{S}}}\!\!\!\!\!\!\!\!\!\! | \left[ \mathbf{F} \right]_{H',R,Q} |.$$
\end{lemma}

\begin{proof} We will first cover the case when we have $k \in \left\{ 1, \dots, r \right\}$ and distinct $i, j \in \left\{ 1, \dots, n_k \right\}$ such that $l_V^{-1} ( L^{(k)}_i ) \cap S^{(k)} \neq \emptyset$ and $l_V^{-1} ( L^{(k)}_j ) \cap S^{(k)} \neq \emptyset$; without loss of generality, let $k=1$, $i=1$, $j=2$. Let $v_1 \in l_V^{-1} ( L^{(1)}_1 ) \cap S^{(1)}$ and $v_2 \in l_V^{-1} ( L^{(1)}_2 ) \cap S^{(1)}$. By separating products of functions depending on whether the edge $e \in E$ contains vertex $v_1$, vertex $v_2$ or none of them and then applying the inequality $|AB| \leq \frac{1}{2\delta}A^2 + \frac{\delta}{2}B^2 \leq \delta^{-1}A^2 + \delta B^2$ for any $A,B \in \mathbb{R}$, we conclude that
$$| \left[ \mathbf{F} \right]_{H,S,Q} | \leq \delta^{-1} \left[ \mathbf{F} \right]_{H',R,Q} + \delta \left[ \mathbf{F} \right]_{H'',R,Q}$$
for labeled hypergraphs $H',H''$ and a tuple of subsets $R$ defined in the following way. Starting with hypergraph $H$, let the label $x_{v_2}$ be a copy of the label $x_{v_1}$, i.e.\@ redefine $l_V(v_2)$ in a way that $l_V(v_2) \in L^{(1)}_1 \backslash l_V(V^{(1)})$ (so that $l_V$ remains an injective function). Also, remove all edges $e \in E$ for which $v_2 \in E$ and add edges $e' \in E$ which have $v_1 \in e'$, but with vertex $v_2$ instead of $v_1$. In analogous way we define labeled hypergraph $H''$. Intuitively, starting from the hypergraph $H$ we removed one of vertices $v_1$ and $v_2$ and then we doubled the remaining vertex and its role. As for the sequence of subsets $R$, we take
$$R^{(1)} = \left\{ v_1, v_2 \right\}, \,\, R^{(k)} = \emptyset, k \geq 2.$$
Notice that $H' \prec H$ as the first different component from the definition of the relation $\prec$ got increased while constructing $H'$. On the other hand, it might happen that the tuple representing the number of times each vertex label appears for the hypergraph $H''$ did not decrease. In that case we will interchange the roles of the vertex labels according to our agreement before this lemma. With that agreement, it may still happen that $H'' \succeq H$ as well as $H'' \prec H$, in which case we use the inequality $\delta < \frac{1}{\delta}$, which is true for any $\delta \in \left< 0,1 \right>$. The claim would then follow by adding the remaining terms $\delta^{-1} \left[ \mathbf{F} \right]_{H',R,Q}$ or $\delta \left[ \mathbf{F} \right]_{H',R,Q}$ to the whole expression. Note that $\mathcal{B}_{H,Q} \equiv 0$ satisfies the first inequality required in the statement of the lemma.

The second case of possible hypergraphs $H$ is when, for each $k \in \left\{ 1, \dots, r \right\}$, there exists at most one $i_k \in \left\{ 1, \dots, m_k \right\}$ such that $l_V^{-1} ( L^{(k)}_{i_k} ) \cap S^{(k)} \neq \emptyset$. Without loss of generality, let $l_V^{-1} ( L^{(1)}_1 ) \cap S^{(1)} \neq \emptyset$; in that case, there exist distinct $v_1,v_2 \in l_V^{-1} ( L^{(1)}_{i_1} ) \cap S^{(1)}$. If we mark $S^{'(1)} := \left\{ v_1, v_2 \right\}$ and $S^{'(k)} = \emptyset$ for $k \geq 2$, we can notice that
$$| \left[ \mathbf{F} \right]_{H,S,Q} | \leq \left[ \mathbf{F} \right]_{H,S',Q}$$
Since it is enough to bound the expression for $S'$, we will assume that $S$ is already defined as $S'$ above. Now, let $\mathcal{B}_{H,Q} := \left[ \mathbf{F} \right]_{H,(\emptyset),Q}.$ Note that the first inequality in the statement of this lemma is satisfied by Lemma \ref{holder} and the normalization of the functions. By Proposition \ref{difference},
$$\square \mathcal{B}_{H,Q} = \sum_{R \in \mathcal{S}} \left[ \mathbf{F} \right]_{H,R,Q}.$$

We will split the family $\mathcal{S}$ into three parts. For each $k \in \left\{ 1, \dots, r \right\}$ we define
\begin{align*}
\mathcal{S}^{(1)} := \big\{ S \in \mathcal{S} :& ( \exists! i_1 \in \mathbb{N} ) l_V^{-1} ( L^{(1)}_{i_1} ) \cap S^{(1)} \neq \emptyset \wedge \left( \forall k \in \left\{ 2, \dots, r \right\} \right) S^{(k)} = \emptyset \big\}, \\
\mathcal{S}^{(2)} := \big\{ S \in \mathcal{S} :& (\exists k \in \{ 2, \dots, r \}) (\exists! i_k \in \mathbb{N}) l_V^{-1} ( L^{(k)}_{i_k} ) \cap S^{(k)} \neq \emptyset \\
&\wedge \left( \forall k' \in \left\{ k+1, \dots, r \right\} \right) S^{(k')} = \emptyset \big\}, \\
\mathcal{S}^{(3)} := \big\{ S \in \mathcal{S} :& (\exists k \in \{ 1, \dots, r \}) (\exists i_k,i_k' \in \mathbb{N}, i_k \neq i_k')\,\, l_V^{-1} ( L^{(k)}_{i_k} ) \cap S^{(k)} \neq \emptyset \\
&\wedge l_V^{-1} ( L^{(k)}_{i_k'} ) \cap S^{(k)} \neq \emptyset \wedge (\forall k' \in \{ k+1, \dots, r \}) S^{(k')} = \emptyset \big\}.
\end{align*}
Notice that $\mathcal{S} = \dot\cup_{k=1}^3 \mathcal{S}^{(k)}$ and that $S \in \mathcal{S}^{(1)}$. Take $R \in \mathcal{S}^{(1)}$; as each of the functions $F_e, e\in E$ is nonnegative, the only possible integration of negative function on a set of positive measure happens each time when the function $\mathbbm{h}^1_I$ is involved, i.e.\@ whenever we include the edge which consists a selected vertex. The only selected vertices appear in the set $S^{(1)}$ and all of them have the label of the form $x_{i_1,j_1}^{(1)}$ for even number of indices $j_1 \in \{ 1, \dots, n_1 \}$. Notice that, no matter which of these variables we use to evaluate the integral expression, by the agreement in (\ref{isti}) and by the agreement of vertices having same first lower indices we can separate the product $\prod_{e \in E} F_e$ into equal products of the form $\prod_{v^{(1)}_{i_1,j_1} \in e \in E} F_e$. Therefore, by changing the order of the variables and separating the integral into more integrals, each of them having only one single variable of the form $x_{i_1,j_1}^{(1)}$, we get a product of same integral which appears an even amount of times. Having the same number to the power of the even natural number, we conclude that the whole expression is nonnegative. This works for any $R \in \mathcal{S}^{(1)}$, therefore
$$\sum_{R \in \mathcal{S}^{(1)}} \left[ \mathbf{F} \right]_{H,R,Q} \geq \left[ \mathbf{F} \right]_{H,S,Q}.$$

Now, let $k \in \left\{ 2, \dots, r \right\}$. For a moment, we will consider a labeled $(r-k+1)$-partite $(r-k+1)$-uniform hypergraph $H_k$ on $\prod_{i=k}^r V^{(i)}$ obtained from $H$ in a way that we keep all vertices from vertex components $V^{(k)}, \dots, V^{(r)}$ with same vertex labels and along with edges which are reduced by removing its vertices from disregarded vertex components $V^{(1)}, \dots, V^{(k-1)}$. Also, if $Q = \prod_{i=1}^r I_i$, define $Q_k := \prod_{i=k}^r I_i$. Along with $S^{'(i)} := \emptyset$ for $i \in \left\{ k+1, \dots, r \right\}$ and for fixed real numbers $\big( (x_{v_i^{(k')}})_{\substack{1 \leq k' \leq k-2 \\ 1 \leq i \leq n_{k'}}} \big)$, we define
$$f_{k-1} ( (x_{v_i^{(k-1)}})_{1 \leq i \leq n_{k-1}} ) := \!\!\!\!\!\!\!\!\!\!\!\!\!\!\!\sum_{\substack{S^{'(k)} \subseteq l_V^{-1} ( L^{(k)}_{i_k} ) \cap S^{(k)} \\ S^{'(k)} \neq \emptyset \,\,\textnormal{and}\,\, |S^{'(k)}| \,\,\textrm{is even}}}\!\!\!\!\!\!\!\!\!\!\!\!\!\!\! \left[ \mathbf{F} \right]_{H_k,(S^{'(i)})_{k \leq i \leq r},Q_k}.$$
The expression in the definition of this function is a sum of integral expressions containing the variables $x_{v_i^{(k')}}$ for each $k' \in \{ 1, \dots, r \}$ and $i \in \{ 1, \dots, n_{k'} \}$, integrating in each variable when $k' \geq k+1$. The function in though of as depending on the independent variables corresponding to $k' = k$ while the other variables for $k' \leq k-1$ are, at this moment, regarded as constants. Similarly as before, this function is nonnegative, so if we apply Lemma \ref{neneg} to function $f_{k-1}$, we can conclude that the function
$$f_{k-2} ( (x_{v_i^{(k-2)}})_{1 \leq i \leq n_{k-2}} ) := \sum_{\substack{S^{'(k-1)} \subseteq V^{(k-1)} \\ |S^{'(k-1)}| \,\,\textrm{is even}}} \sum_{\substack{S^{'(k)} \subseteq l_V^{-1} ( L^{(k)}_{i_k} ) \cap S^{(k)} \\ S^{'(k)} \neq \emptyset \,\,\textnormal{and}\,\, |S^{'(k)}| \,\,\textrm{is even}}}\!\!\!\!\!\!\!\!\!\! \left[ \mathbf{F} \right]_{H_{k-1},(S^{'(i)})_{k-1 \leq i \leq r},Q_{k-1}}$$
is also nonnegative, where $H_{k-1}$ is a $(r-k+2)$-partite $(r-k+2)$-uniform hypergraph on $\prod_{i={k-1}}^r V^{(i)}$ and $Q_{k_1} := \prod_{i=k-1}^r I_i$, defined analogously as $H_k$ and $Q_k$ before. Continuing to apply Lemma \ref{neneg} to each class of variables until we reach last function $f_2$, in variables $x_{v_1^{(1)}}, \dots, x_{v_{n_1}^{(1)}}$, we conclude that
$$\sum_{R \in \mathcal{S}^{(2)}} \left[ \mathbf{F} \right]_{H,R,Q} \geq 0.$$
The case of $R \in \mathcal{S}^{(3)}$ is covered as the first case of this proof, from which follows that
$$| \left[ \mathbf{F} \right]_{H,R,Q} | \leq \delta^{-1} \!\!\!\!\!\!\!\!\!\!\sum_{\substack{H' \in \mathcal{H}_{(n_i)}, H' \prec H \\ R' \in \mathcal{S}}}\!\!\!\!\!\!\!\!\!\! | \left[ \mathbf{F} \right]_{H',R',Q} | + \delta \!\!\!\!\!\!\!\!\!\!\sum_{\substack{H' \in \mathcal{H}_{(n_i)}, H' \succeq H \\ R' \in \mathcal{S}}}\!\!\!\!\!\!\!\!\!\! | \left[ \mathbf{F} \right]_{H',R',Q} |.$$
Combining all these cases, we can conclude that
{\allowdisplaybreaks
\begin{align*}
\left[ \mathbf{F} \right]_{H,S,Q} \leq& \sum_{R \in \mathcal{S}^{(1)}} \left[ \mathbf{F} \right]_{H,R,Q} + \sum_{R \in \mathcal{S}^{(2)}} \left[ \mathbf{F} \right]_{H,R,Q} + \sum_{R \in \mathcal{S}^{(3)}} \left[ \mathbf{F} \right]_{H,R,Q} - \sum_{R \in \mathcal{S}^{(3)}} \left[ \mathbf{F} \right]_{H,R,Q} \\
\leq& \square \mathcal{B}_{H,Q} + \sum_{R \in \mathcal{S}^{(3)}} \bigg( \delta^{-1} \!\!\!\!\!\!\!\!\sum_{\substack{H' \in \mathcal{H}_{(n_i)}, H' \prec H \\ R' \in \mathcal{S}}}\!\!\!\!\!\!\!\! | \left[ \mathbf{F} \right]_{H',R',Q} | + \delta \!\!\!\!\!\!\!\!\sum_{\substack{H' \in \mathcal{H}_{(n_i)}, H' \succeq H \\ R' \in \mathcal{S}}}\!\!\!\!\!\!\!\! | \left[ \mathbf{F} \right]_{H',R',Q} | \bigg) \\
=& \square \mathcal{B}_{H,Q} + C_{(n_i)} \delta^{-1} \!\!\!\!\!\!\!\!\sum_{\substack{H' \in \mathcal{H}_{(n_i)}, H' \prec H \\ R' \in \mathcal{S}}}\!\!\!\!\!\!\!\! | \left[ \mathbf{F} \right]_{H',R',Q} | + C_{(n_i)} \delta \!\!\!\!\!\!\!\!\sum_{\substack{H' \in \mathcal{H}_{(n_i)}, H' \succeq H \\ R' \in \mathcal{S}}}\!\!\!\!\!\!\!\! | \left[ \mathbf{F} \right]_{H',R',Q} |
\end{align*}
}
for $C_{(n_i),\delta} := |\mathcal{S}^{(3)}|$, which is the claim of this lemma.
\end{proof}

\begin{lemma} For every $r$-partite $r$-regular complete hypergraph $H$ and for every $\epsilon \in \left< 0,1 \right>$ there exist an averaging paraproduct-type expression $\mathcal{B}_{H,Q}^{\epsilon}$ satisfying
$$\max_{Q \in \mathcal{T} \cup \mathcal{L}(\mathcal{T})} \mathcal{B}_{H,Q}^{\epsilon} \lesssim_{\left( n_i \right), \epsilon} 1$$
and
$$\sum_{\substack{H' \in \mathcal{H}_{(n_i)}, H' \preceq H \\ S \in \mathcal{S}}}\!\!\!\!\!\!\!\!\!\! | \left[ \mathbf{F} \right]_{H',S,Q} | \leq \square \mathcal{B}_{H,Q}^{\epsilon} + \epsilon \!\!\!\!\!\!\!\!\!\!\sum_{\substack{H' \in \mathcal{H}_{(n_i)}, H' \succ H \\ S \in \mathcal{S}}}\!\!\!\!\!\!\!\!\!\! | \left[ \mathbf{F} \right]_{H',S,Q} |.$$
\end{lemma}

\begin{proof} As the totally ordered set $(\mathcal{H}_{(n_i)}, \preceq)$ is finite, we will prove this claim by induction over the hypergraphs from this family. Before we begin, let $H \in \mathcal{H}_{(n_i)}$ be arbitrary non-maximal hypergraph and let $H_s$ be an immediate successor of $H$. Let $C_{(n_i)}$ be as in Lemma \ref{bound}. Suppose that there exists an averaging paraproduct-term $\mathcal{B}^{\epsilon'}_{H,Q}$ such that
\begin{equation}\sum_{\substack{H' \in \mathcal{H}_{(n_i)}, H' \preceq H \\ S \in \mathcal{S}}} | \left[ \mathbf{F} \right]_{H',S,Q} | \leq \square \mathcal{B}_{H,Q}^{\epsilon'} + \left( \frac{\epsilon}{4C_{(n_i)}|\mathcal{H}_{(n_i)}||V|} \right)^2 \sum_{\substack{H' \in \mathcal{H}_{(n_i)}, H' \succ H \\ S \in \mathcal{S}}} | \left[ \mathbf{F} \right]_{H',S,Q} |, \label{ind}\end{equation}
where $\epsilon' := \left( \frac{\epsilon}{4C_{(n_i)}|\mathcal{H}_{(n_i)}||V|} \right)^2$ and $\epsilon \in \left< 0,1 \right>$ is arbitrary. Applying Lemma \ref{bound} for every $H' \in \mathcal{H}_{(n_i)}, H' \leq H_s$ and $\delta := \frac{\epsilon}{4C_{(n_i)}|\mathcal{H}_{(n_i)}| |V|}$, we have
{\allowdisplaybreaks
\begin{align}
\!\!\!\!\!\!\!\!\!\!\sum_{\substack{H' \in \mathcal{H}_{(n_i)}, H' \preceq H_s \\ S \in \mathcal{S}}}\!\!\!\!\!\!\!\!\!\! | \big[ \mathbf{F} &\big]_{H',S,Q} | \leq \!\!\!\!\!\!\!\!\!\!\sum_{\substack{H' \in \mathcal{H}_{(n_i)}, H' \preceq H_s \\ S \in \mathcal{S}}}\!\!\!\!\!\!\!\!\!\! \square \mathcal{B}_{H',Q} + \frac{4C_{(n_i)}^2|\mathcal{H}_{(n_i)}|^2|V|^2}{\epsilon} \!\!\!\!\!\!\!\!\!\!\sum_{\substack{H'' \in \mathcal{H}_{(n_i)}, H'' \preceq H \\ R \in \mathcal{S}}}\!\!\!\!\!\!\!\!\!\! | \left[ \mathbf{F} \right]_{H'',R,Q} | \nonumber\\
&+ \frac{\epsilon}{4} \sum_{\substack{H'' \in \mathcal{H}_{(n_i)} \\ R \in \mathcal{S}}} | \left[ \mathbf{F} \right]_{H'',R,Q} | \nonumber\\
\stackrel{(\ref{ind})}{\leq}& \!\!\!\!\!\!\!\!\!\!\sum_{\substack{H' \in \mathcal{H}_{(n_i)}, H' \preceq H_s \\ S \in \mathcal{S}}}\!\!\!\!\!\!\!\!\!\! \square \mathcal{B}_{H',Q} + \frac{4C_{(n_i)}^2|\mathcal{H}_{(n_i)}|^2|V|^2}{\epsilon} \square \mathcal{B}_{H,Q}^{\epsilon'} + \frac{\epsilon}{2} \sum_{\substack{H' \in \mathcal{H}_{(n_i)} \\ S \in \mathcal{S}}} | \left[ \mathbf{F} \right]_{H',S,Q} | \nonumber\\
\leq& \!\!\!\!\!\!\!\!\!\!\sum_{\substack{H' \in \mathcal{H}_{(n_i)}, H' \preceq H_s \\ S \in \mathcal{S}}}\!\!\!\!\!\!\!\!\!\! \square \mathcal{B}_{H',Q} + \frac{4C_{(n_i)}^2|\mathcal{H}_{(n_i)}|^2|V|^2}{\epsilon} \square \mathcal{B}_{H,Q}^{\epsilon'} + \frac{\epsilon}{2} \!\!\!\!\!\!\!\!\!\!\sum_{\substack{H' \in \mathcal{H}_{(n_i)}, H' \succ H_s \\ S \in \mathcal{S}}}\!\!\!\!\!\!\!\!\!\! | \left[ \mathbf{F} \right]_{H',S,Q} | \nonumber\\
&+ \frac{1}{2} \!\!\!\!\!\!\!\!\!\!\sum_{\substack{H' \in \mathcal{H}_{(n_i)}, H' \preceq H_s \\ S \in \mathcal{S}}}\!\!\!\!\!\!\!\!\!\! | \left[ \mathbf{F} \right]_{H',S,Q} | \label{raspis}.
\end{align}
}where we used $|\mathcal{S}| \leq |V|$. Moving the last sum on the left side of the inequality and multiplying the inequality by $2$, we get
\begin{equation}\sum_{\substack{H' \in \mathcal{H}_{(n_i)}, H' \preceq H_s \\ S \in \mathcal{S}}}\!\!\!\!\!\!\!\!\!\! | \left[ \mathbf{F} \right]_{H',S,Q} | \leq \square \mathcal{B}^{\epsilon}_{H_s,Q} + \epsilon \!\!\!\!\!\!\!\!\!\!\sum_{\substack{H' \in \mathcal{H}_{(n_i)}, H' \succ H_s \\ S \in \mathcal{S}}}\!\!\!\!\!\!\!\!\!\! | \left[ \mathbf{F} \right]_{H',S,Q} |, \label{step}\end{equation}
with additional notation
$$\mathcal{B}^{\epsilon}_{H_s,Q} := 2\!\!\!\!\!\!\!\!\!\!\sum_{\substack{H' \in \mathcal{H}_{(n_i)}, H' \preceq H_s \\ S \in \mathcal{S}}}\!\!\!\!\!\!\!\!\!\! \mathcal{B}_{H',Q} + \frac{8C_{(n_i)}^2|\mathcal{H}_{(n_i)}|^2|V|^2}{\epsilon} \mathcal{B}_{H,Q}^{\epsilon'},$$
which is an averaging paraproduct-type expression.

Now we proceed to the induction. The induction basis for the minimal hypergraph $H_m$ is actually (\ref{step}) with $H_s = H_m$ and it follows from (\ref{raspis}), where, instead of (\ref{ind}) (we cannot refer to it as $H_m$ does not have preceding elements), we use a trivial inequality
$$0 \leq \square \mathcal{B}_{H_m,Q}^{\epsilon'} + \left( \frac{\epsilon}{4C_{(n_i)}|\mathcal{H}_{(n_i)}||V|} \right)^2 \sum_{\substack{H' \in \mathcal{H}_{(n_i)}}} | \left[ \mathbf{F} \right]_{H',S,Q} |$$
for $ \mathcal{B}_{H_m,Q}^{\epsilon'} := 0$, which trivially satisfies the required bound. Suppose that the claim of the lemma is satisfied for certain $H \in \mathcal{H}_{(n_i)}$, i.e.\@ we have (\ref{ind}). Then the same claim follows from its successor $H_s$, which is actually (\ref{step}), with $\mathcal{B}^{\epsilon}_{H_s,Q}$, which also satisfies the required bound by mathematical induction. With this, the required mathematical induction is complete.
\end{proof}

\begin{lemma} \label{noeps} For every $r$-partite $r$-regular complete hypergraph $H$ and for each $S \in \mathcal{S}$ there exist an averaging paraproduct-type expression $\mathcal{B}_{H,(\emptyset)}$ satisfying
$$\max_{Q \in \mathcal{T} \cup \mathcal{L}(\mathcal{T})} \mathcal{B}_{H,Q} \lesssim_{\left( n_i \right)} 1$$
and
$$\sum_{\substack{H' \in \mathcal{H}_{(n_i)} \\ S \in \mathcal{S}}} | \left[ \mathbf{F} \right]_{H',S,Q} | \leq \square \mathcal{B}_{H,Q}.$$
\end{lemma}
\begin{proof} As discussed while defining the totally ordered set $(\mathcal{H}_{(n_i)}, \preceq)$, there exists a maximal hypergraph $H_M$. The claim of this lemma follows from previous lemma by applying it for any fixed $\epsilon \in \left< 0,1 \right>$ and then by using $\mathcal{B}_{H,Q} := \mathcal{B}_{H_M,Q}^{\epsilon}$
\end{proof}

For each tuple of functions $\mathbf{F}$ and each finite convex tree $\mathcal{T}$ we define
$$\Lambda_{\mathcal{T}} \left( \mathbf{F} \right) := \sum_{Q \in \mathcal{T}} \left| Q \right| \left[ \mathbf{F} \right]_{H,S,Q}$$
where $H = (V,E)$ is any $r$-partite $r$-uniform labeled hypergraph and $S = (S^{(i)})_{1\leq i \leq r}$ is a tuple such that $S^{(i)} \subseteq V^{(i)}$ for each $i \in \left\{ 1, \dots, r \right\}$ and there exists $i_0 \in \left\{ 1, \dots, r \right\}$ such that $|S^{(i_0)}| \geq 2$.

\begin{lemma} \label{enough} Let $H=(V,E)$ be a $r$-regular $r$-uniform complete labeled hypergraph and let $\mathcal{T}$ be a finite convex tree. Suppose that for each $Q \in \mathcal{T}$ there exists an averaging paraproduct-type term $\mathcal{B}_{H,Q}$ such that
$$|[\mathbf{F}]_{H,S,Q}| \leq \square \mathcal{B}_{H,Q} \,\,\, \textrm{and} \,\,\, \max_{Q \in \mathcal{T} \cup \mathcal{L}(\mathcal{T})} \mathcal{B}_{H,Q} \lesssim_{\left( n_i \right)} 1.$$
Then,
$$| \Lambda_{\mathcal{T}} ( \mathbf{F} ) | \lesssim_{\left( n_i \right)} |Q_{\mathcal{T}}|.$$
\end{lemma}

\begin{proof}We have
\begin{align*}| \Lambda_{\mathcal{T}} ( \mathbf{F} ) | &\leq \sum_{Q \in \mathcal{T}} \left| Q \right| | \left[ \mathbf{F} \right]_{H,S,Q} | \leq \sum_{Q \in \mathcal{T}} \left| Q \right| \square \mathcal{B}_{H,Q} \\
&= \sum_{Q \in \mathcal{T}} \bigg( \sum_{Q' \in \mathcal{C}(Q)}|Q'|\mathcal{B}_{H,Q'} - |Q|\mathcal{B}_{H,Q} \bigg) \\
&= \sum_{Q \in \mathcal{L}(\mathcal{T})} |Q|\mathcal{B}_{H,Q} - |Q_{\mathcal{T}}|\mathcal{B}_{H,Q_{\mathcal{T}}} \lesssim_{\left( n_i \right)}  \sum_{Q \in \mathcal{L}(\mathcal{T})} |Q| \leq |Q_{\mathcal{T}}|,
\end{align*}
where we also used that the averaging paraproduct-type term, given nonnegative functions $\mathbf{F}$, is also nonnegative.
\end{proof}

\begin{proposition} \label{tree}
Let $H=(V,E)$ be a $r$-regular $r$-uniform labeled hypergraph such that its label function $l_E$ is injective and, more explicitly, $l_E(e)=F_e$ for each $e\in E$. For any finite convex tree $\mathcal{T}$ with root $Q_{\mathcal{T}}$ we have
$$|\Lambda_{\mathcal{T}} \left( \mathbf{F} \right)| \lesssim_{(n_i)} \left| Q_{\mathcal{T}} \right| \prod_{e \in E} \max_{Q \in \mathcal{T} \cup \mathcal{L} \left( \mathcal{T} \right)} [ F_e^{d_e} ]_Q^{\frac{1}{d_e}}.$$
\end{proposition}

\begin{proof} First, we will prove the proposition in the special case when $E = \prod_{i=1}^r \{ v^{(i)}_1, \dots, v^{(i)}_{n_i} \}$, i.e.\@ for a complete $r$-uniform hypergraph. In that case this number is same for each edge $e \in E$. First, notice that it will be enough to prove the claim of the proposition with additional assumptions
$$\left| Q_{\mathcal{T}} \right| = 1\,\,\,\mathrm{and}\,\,\,\max_{Q \in \mathcal{T} \cup \mathcal{L} \left( \mathcal{T} \right)} [ F_e^{d_e} ]_Q^{\frac{1}{d_e}} = 1 \,\,\,\textrm{for each}\,\,\, e \in E,$$
in which case we need to prove
$$\Lambda_{\mathcal{T}} \left( \mathbf{F} \right) \lesssim_{(n_i)} 1.$$
We are required to dominate each term $\left[ \mathbf{F} \right]_{H,S,Q}$, $Q \in \mathcal{T}$, from the definition of $\Lambda_{\mathcal{T}}$. First, notice that we do not necessarily have $S \in \mathcal{S}$. However, if we, without loss of generality, suppose that $\max_{i \in \mathbb{N}} |S^{(1)} \cap l_V^{-1}(L^{(1)}_i) | \geq 2$ and take $v_{i_1} \in S^{(1)} \cap l_V^{-1}(L^{(1)}_{i_1}),v_{i_1} \in S^{(2)} \cap l_V^{-1}(L^{(1)}_{i_2})$ for some $i_1 \neq i_2$, then using the Cauchy-Schwarz inequality we obtain
$$\left[ \mathbf{F} \right]_{H,S,Q} \leq \frac{1}{2}\left[ \mathbf{F} \right]_{H_1,S_1,Q} + \frac{1}{2}\left[ \mathbf{F} \right]_{H_2,S_2,Q}$$
for hypergraphs $H_1$ and $H_2$ and tuples of selected vertices $S_1$ and $S_2$ defined in the following way. For each $j,j' \in \left\{ 1,2 \right\}, j \neq j'$, a hypergraph $H_j$ has the label function $l_V^j \big|_{V \backslash \{ v_{i_{j'}} \}} := l_V \big|_{V \backslash \{ v_{i_{j'}} \}}$ and $l_V^j(v_{i_{j'}}) := l_V(v_{i_j})$. Also, $S^{(1)}_1 = S^{(1)}_2 := \{ v_1, v_2 \}$ and $S^{(i)}_1 = S^{(i)}_2 := \emptyset$ for $i \in \{ 2, \dots, r \}$. We can see that $S_1,S_2 \in \mathcal{S}$ for each choice of $j$. For $Q \in \mathcal{C}_r$, let $\mathcal{B}_{H,Q}$ be as in Lemma \ref{noeps}. Applying Lemma \ref{enough} and using the bound from Lemma \ref{noeps}, we have
$$\Lambda_{\mathcal{T}} ( \mathbf{F} ) \leq \frac{1}{2} \sum_{Q \in \mathcal{T}} |Q| \big( \left[ \mathbf{F} \right]_{H_1,S_1,Q} +\left[ \mathbf{F} \right]_{H_2,S_2,Q} \big)  \lesssim_{(n_i)} |Q_{\mathcal{T}}| = 1.$$

Now suppose that we are given an arbitrary set of edges $E$. It might happen that the hypergraph $H$ contains isolated vertices, i.e.\@ those that are not elements of any edge. If $v$ is an isolated vertex, then, by the definition of $[\mathbf{F}]_{H,S,Q}$ and the injectivity of $l_V$, the variable $l_V(v)$ will appear either in the expression $\mathbbm{h}_I^1(l_V(v))$ or in $\mathbbm{h}_I^0(l_V(v))$ for some $I \in \mathcal{C}_1$. In first case, integrating by that variable we get $[\mathbf{F}]_{H,S,Q} = 0$, while in the other case, since the integral of function $\mathbbm{h}_{I_k}$ equals one, the expression remains the same if we leave out that variable (and the vertex) from the expression. Therefore, isolated vertices give no significant contribution to the expression to $\Lambda_{\mathcal{T}}$, so we may assume that there exist $k \in \mathbb{N}$ and connected components $\prod_{i=1}^r V^{(i)}_j$ for each $j \in \{ 1, \dots, k \}$ with no isolated vertices. Notice that each number $d_e$ depends on which of the components the edge $e$ belongs to, so we will also denote that number as $d^{(j)}$, where $j \in \{ 1, \dots, k \}$ is such that $e \in \prod_{i=1}^r V^{(i)}_j$. We can also suppose that these $k$ components form complete $r$-partite $r$-uniform graphs by adding missing edges from the set $\cup_{j=1}^k \prod_{i=1}^r V^{(i)}_j$ and, for those edges $e$, defining $F_e \equiv 1$. For each $j \in \{ 1, \dots, k \}$, let $H_1, \dots, H_k$ be the $r$-partite $r$-uniform complete hypergraphs representing connected components of the hypergraph $H$; also, let $S_j = (S_j^{(i)})_{1 \leq i \leq r}$ be defined as $S_j^{(i)} := S^{(i)} \cap V^{(i)}_j$ for each $i \in \{ 1, \dots, r \}$ and $j \in \{ 1, \dots, k \}$. With the additional notation of $\mathbf{F}_E = \mathbf{F} = ( F_e )_{e \in E}$ and $\mathbf{F}_{E_j} = ( F_e )_{e \in E_j}$ for each $j = 1, \dots, k$ we can notice that
$$\Lambda_{\mathcal{T}} ( \mathbf{F}_E ) = \sum_{Q \in \mathcal{T}} |Q| \prod_{j=1}^k [\mathbf{F}_{E_j}]_{H_j,S_j,Q}.$$
The first case is when there exists $j \in \{ 1, \dots, k \}$ such that $|S^{(1)}_j| \geq 2$. We can apply this proposition to the hypergraph $H_j$ as it belongs to the first case that we already covered. Therefore
$$\sum_{Q \in \mathcal{T}} |Q| [\mathbf{F}_{E_j}]_{H_j,S_j,Q} \lesssim_{(n_i)} \left| Q_{\mathcal{T}} \right| \prod_{e \in E_j} \max_{Q \in \mathcal{T} \cup \mathcal{L} \left( \mathcal{T} \right)} [ F_e^{d^{(j)}} ]_Q^{\frac{1}{d^{(j)}}} = 1.$$
As for each $Q \in \mathcal{T}$ and each $j' \in \{ 1, \dots, k \} \backslash \{ j \}$, applying Lemma \ref{holder} we get
$$| [\mathbf{F}_{E_{j'}}]_{H_{j'},S_{j'},Q} | \leq [\mathbf{F}_{E_{j'}}]_{H_{j'},(\emptyset),Q} \leq \prod_{e \in E_{j'}} [ F_e^{d^{(j')}} ]_Q^{\frac{1}{d^{(j')}}} \leq 1.$$
It follows that
\begin{align*}
\Lambda_{\mathcal{T}} ( \mathbf{F}_E ) = \sum_{Q \in \mathcal{T}} |Q| [\mathbf{F}_{E_j}]_{H_j,S_j,Q} \prod_{\substack{1 \leq j' \leq k \\ j' \neq j}} [\mathbf{F}_{E_{j'}}]_{H_{j'},S_{j'},Q} \lesssim_{(n_i)} 1,
\end{align*}
which proves the claim of this proposition. The second case is when there exist $j_1,j_2 \in \{ 1, \dots, k \}$ such that $S^{(1)}_{j_1} \neq \emptyset \neq S^{(1)}_{j_2}$; without loss of generality, let $j_1 = 1$ and $j_2 = 2$. Using Lemma \ref{holder} in similar way as above, we can observe that
\begin{align*}| \left[ \mathbf{F}_E \right]_{H,S,Q} | &\leq | [\mathbf{F}_{E_1}]_{H_1,S_1,Q} || [\mathbf{F}_{E_2}]_{H_2,S_2,Q} | \leq \frac{1}{2}[\mathbf{F}_{E_1}]_{H_1,S_1,Q}^2 + \frac{1}{2}[\mathbf{F}_{E_2}]_{H_2,S_2,Q}^2.\end{align*}
By changing the roles of the vertices let us assume that $V^{(i)}_1 = \{ v_1^{(i)}, \dots, v_{l_i}^{(i)} \}$ for each $i \in \{ 1, \dots, r \}$ and that $v^{(1)}_1 \in S^{(1)}_1$. If $d^{(1)} = 1$, i.e.\@ if $l_1 = \dots = l_r = 1$, then, for $\mathcal{B}_{H_1,Q} := [ F_{(1,\dots,1)} ]^2_{H_1,(\emptyset),Q}$ we have
$$\square \mathcal{B}_{H_1,Q} = \sum_{\substack{R^{(1)} \subseteq \{ v_1^{(1)} \} \\ \dots \\ R^{(r)} \subseteq \{ v_1^{(r)} \} \\ R = (R^{(i)}) \neq (\emptyset) }} [ F_{(1,\dots,1)} ]^2_{H_1,R,Q} \geq [ F_{(1,\dots,1)} ]^2_{H_1,S,Q}.$$
Note as well that
$$\max_{Q \in \mathcal{T} \cup \mathcal{L}(\mathcal{T})} \mathcal{B}_{H_1,Q}(F_{(1,\dots,1)}) = \big( \max_{Q \in \mathcal{T} \cup \mathcal{L}(\mathcal{T})} [ F_{(1,\dots,1)} ]_{H_1,(\emptyset),Q} \big)^2 = 1.$$
Analogously, we construct $\mathcal{B}_{H_2,Q} := [ F_{(1,\dots,1)} ]^2_{H_2,(\emptyset),Q}$. The proof of the proposition is complete in this case after we apply Lemma \ref{enough} with $\mathcal{B}_{H,Q} := \frac{1}{2}\mathcal{B}_{H_1,Q}+\frac{1}{2}\mathcal{B}_{H_2,Q}$. In the case $l_1 = 1$ and $l_2 \geq 2$, using Jensen's inequality for the convex function $x \longmapsto x^2$ and the integral of type $\int_Q \frac{1}{|Q|}d\mathbbm{x}$, we have that
$$[\mathbf{F}_{E_1}]_{H_1,S_1,Q}^2 \leq [\mathbf{F}_{E_1'}]_{H_1',S_1',Q}.$$
Here, $H_1'$ is the $r$-partite $r$-uniform complete hypergraph with set of vertices $V' := V \cup \{ v_2^{(1)} \}$, set of edges
\[ E_1' := E_1 \cup \{ \{ v_2^{(1)} \} \cup (e \backslash \{ v_1^{(1)} \}) : v_1^{(1)} \in e \in E_1 \}\]
and the label function $l_{V'}$ given with $l_{V'} \big|_{V} := l_V$, while the value $l_{V'}(v_2^{(1)})$ can be chosen as an arbitrary copy of $x_{v^{(1)}_1}$, as long as $l_{V'}$ is an injective function. Also, $S_1^{'(1)} := \{ v_1^{(1)}, v_2^{(1)} \}$ and $S_1^{'(i)} := \emptyset$ for $i \in \{ 2, \dots, r \}$; in short, we copied the single vertex $v_1^{(1)}$ from the first part of the $r$-partition along with the edges that contain that vertex and selected only those two vertices ($v_1^{(1)}$ with its copy) out of all vertices in the hypergraph. Notice that $S_1' \in \mathcal{S}$; we can apply Lemma \ref{enough} with $\mathcal{B}_{H_1',Q}$ that we get from Lemma \ref{noeps}. It is important to notice that the numbers of vertices in each of the partition sets of the hypergraphs $H_1'$ and $H_2'$ have changed, therefore affecting the exponents $d_e$ and possibly changing the range of possible exponents $p_e$ while applying Lemma \ref{enough}. However, this is not the case as we only increased $l_1$ by one (when adding $v_2^{(1)}$) and $d_e > l_1$, so the maximum from the definition of that exponent remains the same.

The remaining case is when $l_1 \geq 2$. First we can bound
$$[\mathbf{F}_{E_1}]_{H_1,S_1,Q}^2 \leq [\mathbf{F}_{E_1}]_{H_1,S_1',Q}^2,$$
in a way that $S_1´ = ( \{ v_1^{(1)} \}, \emptyset, \dots, \emptyset )$. Then we can group the integral expression depending on whether any function $F_e$ or any of the Haar functions appear to be evaluated in the cancellative variable $x_{v_1^{(1)}}$, the non-cancellative variable $x_{v_2^{(1)}}$ or if it has none of these two variables. Then, by the application of arithmetic-geometric inequality and also by bounding the complete non-cancellative integral expression with $1$, we get
$$[\mathbf{F}_{E_1}]_{H_1,S_1',Q}^2 \leq [\mathbf{F}_{E_1''}]_{H_1'',S_1'',Q}.$$
This time, $H_1''$ is the $r$-partite $r$-uniform complete hypergraph with a set of vertices $V'' := ( V \cup \{ v_1^{'(1)} \}) \backslash \{ v_2^{(1)} \}$, a set of edges $E_1''$ given with $E_1'' := \{ e \in E_1 : v_2^{(1)} \notin e \} \cup \{ \{ v_1^{'(1)} \} \cup (e \backslash \{ v_1^{(1)} \}) : v_1^{(1)} \in e \in E_1 \} $ and the label function $l_{V''}$ such that $l_{V''} \big|_{V} := l_V$ and $l_{V''}(v_1^{'(1)})$ is a copy of $x_1^{(1)}$, in a way that $l_{V'}$ is still an injective function. With that, $S_1^{''(1)} := \{ v_1^{(1)}, v_1^{'(1)} \}$ and $S_1^{''(i)} := \emptyset$ for $i \in \{ 2, \dots, r \}$. In this case we copied the vertex $v_1^{(1)}$ along with the edges that contain it and left off $v_2^{(1)}$ with each edge that might contain it. The selected vertices are only the first, already selected, vertex $v_1^{(1)}$ along with its new copy $v_1^{'(1)}$. Note that, again, $S_1'' \in \mathcal{S}$, so we use Lemma \ref{noeps} to get $\mathcal{B}_{H_1'',Q}$ and then the result of the proposition follows by applying Lemma \ref{enough} again. As in the previous case, we can notice that the lemma is applied for the same number $d_e$ as the number of vertices in each of the partition sets remains unchanged.
\end{proof}

\section{Decomposition into entangled dyadic paraproducts} \label{cetvrtaprava}

Let us introduce the notation of the \emph{elementary tensor product}, which, for two functions $f,g : \mathbb{R} \rightarrow \mathbb{C}$, is denoted and defined as
$$(f \otimes g) (x,y) := f(x)g(y) \,\, \textnormal{for each} \,\, x,y \in \mathbb{R}.$$
By the associativity of the operation $\otimes$, we will assume the notation $f_1 \otimes f_2 \otimes \dots \otimes f_m$ as an elementary tensor product of more than two functions and also write $\otimes_{i=1}^m f_i$. For what follows, we consider all functions of the form
$$\mathbbm{h}_Q^S := | Q |^{\frac{1}{2}} \displaystyle \bigotimes_{k=1}^r \bigg( \bigotimes_{v^{(k)}_i \in S^{(i)}} \mathbbm{h}_{I_i^{(k)}}^1 \bigg) \bigg( \bigotimes_{v^{(k)}_i \in ( S^{(i)} )^c} \mathbbm{h}_{I_i^{(k)}}^0 \bigg),$$
where $Q = \prod_{i=1}^r \prod_{j=1}^{n_i} I_j^{(i)} \in \mathcal{C}_n$ is arbitrary and $S = ( (S^{(k)})_{k=1}^r ) \neq (\emptyset)$ is an $r$-tuple of selected vertices from the $r$-partitioned set of vertices, i.e.\@ $S^{(k)} \subseteq V^{(k)}$ for each $k \in \{ 1, \dots, r \}$. Notice that these are the tensor products of $\textnormal{L}^2$-normalized Haar functions with at least one of them being cancellative. This means that for a perfect dyadic Calder\' on-Zygmund kernel $K$, being a square-integrable function over $\mathbb{R}^n$, we have
\begin{equation}K = \sum_{\substack{S = (S^{(i)})_{i=1}^r \\ \left( \forall i \in \left\{ 1, \dots, r \right\} \right) S^{(i)} \subseteq V^{(i)}\\ \left( \exists i_0 \in \left\{ 1, \dots, r \right\} \right) \left| S^{(i_0)} \right| \neq 0}} \sum_{Q = \prod_{i=1}^r \prod_{j=1}^{n_i} I^{(i)}_{j_i} \in \mathcal{C}_n} \left< K, \mathbbm{h}_Q^S \right>_{\textnormal{L}^2 (\mathbb{R}^n)}  \mathbbm{h}_Q^S.\label{kernel}\end{equation}
Notice that, as $K$ is constant on dyadic cubes not intersecting the diagonal and each of these tensor products has a cancellation in at least one of the variables, the corresponding scalar products equal zero, so we can actually consider this sum only over dyadic cubes $Q = \prod_{i=1}^r \prod_{j=1}^{n_i} I^{(i)}_{j_i}$ for which $I^{(i)}_{j_1} = I^{(i)}_{j_2}$ for each $j_1,j_2 \in \left\{ 1, \dots, n_i \right\}$ and $i \in \left\{1, \dots, r \right\}$. Using this, we can present the form $\Lambda_E$ as
$$\Lambda_E \left( \mathbf{F} \right) = \!\!\!\!\!\!\!\!\!\!\sum_{\substack{S = (S^{(i)})_{i=1}^r \\ \left( \forall i \in \left\{ 1, \dots, r \right\} \right) S^{(i)} \subseteq V^{(i)}\\ \left( \exists i_0 \in \left\{ 1, \dots, r \right\} \right) \left| S^{(i_0)} \right| \neq 0}}\!\!\!\!\!\!\!\!\!\! \sum_{Q = \prod_{i=1}^r \left( I^{(i)} \right)^{n_i} \in \mathcal{C}_n}\!\!\!\!\!\!\!\!\!\! \left< K, \mathbbm{h}_Q^S \right>_{\textnormal{L}^2 (\mathbb{R}^n)}  \int_{\mathbb{R}^m} \bigg( \prod_{e \in E} F_e ( \mathbbm{x}_e ) \bigg) \mathbbm{h}_Q^S (\mathbbm{x}) d\mathbbm{x},$$
using the assumption that functions $F_e, e \in E$ and $K$ are bounded and compactly supported. Therefore the expression under the integral is absolutely integrable, so we can use the Lebesgue dominated convergence theorem. Notice that for the proof of Theorem \ref{T1thm}(a) it is enough to show that the expression
$$\Lambda_E^S = \Lambda_E^S \left( \mathbf{F} \right) := \sum_{Q = \prod_{i=1}^r \left( I^{(i)} \right)^{n_i} \in \mathcal{C}_n} \left< K, \mathbbm{h}_Q^S \right>_{\textnormal{L}^2 (\mathbb{R}^n)}  \int_{\mathbb{R}^m} \bigg( \prod_{e \in E} F_e ( \mathbbm{x}_e ) \bigg) \mathbbm{h}_Q^S (\mathbbm{x}) d\mathbbm{x}$$
also satisfies inequality (\ref{T1bound}); we will call these expressions \emph{entangled dyadic paraproducts}. Another useful way of writing this form will be
$$\Lambda_E^S \left( \mathbf{F} \right) = \sum_{Q = \prod_{i=1}^r \left( I^{(i)} \right)^{n_i} \in \mathcal{C}_n} |Q| \lambda_Q \left[ \mathbf{F} \right]_{H,S,Q},$$
with $\lambda_Q$ defined as $\lambda_Q := |Q|^{-\frac{1}{2}} \big< K, \mathbbm{h}_Q^S \big>.$ This time, the evaluation of a tuple $\mathbf{F}$ in hypergraph $H$ is defined a bit differently, having a Calder\'{o}n-Zygmund kernel instead of Haar functions in the integral expression.

Denote $\mathcal{I}_S := \{ i \in \{ 1, \dots, r \} : S^{(i)} \neq \emptyset \}$. The form $\Lambda_E^S$ will be called \emph{cancellative} if either
\begin{enumerate}
	\item[(C1)] $\max\limits_{1 \leq i \leq r} |S^{(i)}| \geq 2$, or
    \item[(C2)] $\max\limits_{1 \leq i \leq r} |S^{(i)}| = 1$ and there does not exist $l \in \{ 1, \dots, k \}$ such that $\cup_{1 \leq i \leq r} S^{(i)} \subseteq V_l$.
\end{enumerate}
Otherwise, it is \emph{non-cancellative} if
\begin{enumerate}
	\item[(NC)] $\max\limits_{1 \leq i \leq r} |S^{(i)}| = 1$ and there exists $l \in \{ 1, \dots, k \}$ such that $\cup_{1 \leq i \leq r} S^{(i)} \subseteq V_l$.
\end{enumerate}
We can consider the cancellative form as the one consisting of (at least) two different variables that bring cancellation to the whole expression, but are not entangled in any way (so that those cancellations do not depend on or influence each other).

In the proof of the following proposition, again, first we are going to prove a certain bound locally, by taking the sum only over the dyadic cubes belonging to the certain finite convex tree $\mathcal{T}$. Therefore we define the localized version of the form $\Lambda_E^S$ as
    \begin{equation}\Lambda_{E,\mathcal{T}}^S \left( \mathbf{F} \right) := \sum_{Q  \in \mathcal{T}} |Q| \lambda_Q \left[ \mathbf{F} \right]_{H,S,Q}.\label{localtree} \end{equation}
   
Strictly speaking, we are slightly abusing the notation $\lambda_Q$, as it this coefficient is sometimes associated with a dyadic cube in $\mathbb{R}^r$ and sometimes with the corresponding ``diagonal'' dyadic cube in $\mathbb{R}^n$.

\begin{proposition} \label{canc} Let $\Lambda_E^S$ be a cancellative entangled dyadic paraproduct.
\begin{enumerate}[label=(\alph*)]
	\item If (\ref{first}) holds, then for the corresponding coefficients $\lambda = (\lambda_Q)_{Q \in \mathcal{C}_r}$ we have
    $$\left\| \lambda \right\|_{\ell^{\infty}(\mathcal{C}_r)} \lesssim 1.$$
    \item For a finite convex tree $\mathcal{T}$ and a localized cancellative entangled dyadic paraproduct $\Lambda_{E,\mathcal{T}}^S$ we have
    $$|\Lambda_{E,\mathcal{T}}^S \left( \mathbf{F} \right)| \lesssim \left\| \lambda \right\|_{\ell^{\infty}(\mathcal{C}_r)} |Q_{\mathcal{T}}| \prod_{e \in E} \max_{Q \in \mathcal{T} \cup \mathcal{L} \left( \mathcal{T} \right)} [ F_e^{d_e} ]_Q^{\frac{1}{d_e}}.$$
\end{enumerate}
\end{proposition}

\begin{proof} (a) Let $Q = \prod_{i=1}^r I^{(i)} \in \mathcal{C}_r$. If we take $F_e := \mathbbm{1}_Q = \otimes_{i=1}^r \mathbbm{1}_{I^{(i)}}$ for each $e \in E$, then our form $\Lambda_E$ takes the form
\begin{equation}| \left< K, \otimes_{i=1}^r \otimes_{j=1}^{n_i} \mathbbm{1}_{I^{(i)}} \right>_{\textnormal{L}^2 (\mathbb{R}^n)} | = | \Lambda_E ( \mathbf{F} ) | \stackrel{(\ref{first})}{\lesssim} |Q|. \label{same}\end{equation}
Notice that both the cancellative and the non-cancellative Haar function can be written in the form $\frac{1}{|I|} \left( \mathbbm{1}_{I_L} \pm \mathbbm{1}_{I_R} \right)$ and, as left and right halves of each dyadic interval are mutually disjoint, we can bound $\lambda_Q$ as
$$| \lambda_Q | \leq |Q|^{-1} \sum_{\substack{1 \leq i \leq r \\ 1 \leq j \leq n_i}} \sum_{I^{(i)}_j \in \{ I^{(i)}_L, I^{(i)}_R \}} | \big< K, \otimes_{i=1}^r \otimes_{j=1}^{n_i} \mathbbm{1}_{I^{(i)}_j} \big>_{\textnormal{L}^2 (\mathbb{R}^n)} |.$$
In the $2^r$ cases when $I^{(i)}_1 = \dots = I^{(i)}_{n_i}$ for each $i \in \left\{ 1, \dots, r \right\}$ we can apply (\ref{same}) to obtain boundedness of each summand by a constant. To show the same bound for the remaining cases we can without loss of generality assume that, for a certain $k \in \left\{ 1, \dots, r \right\}$, we have $I^{(i)}_1 = I^{(i)}_L$ and $I^{(i)}_2 = I^{(i)}_R$ for each $i \in \left\{ 1, \dots, k \right\}$ and $I^{(i)}_1 = \dots = I^{(i)}_{n_i}$ for each $i \in \{ k+1, \dots, r \}$. Let $x^{(i)}_0$ be a common endpoint of $I^{(i)}_L$ and $I^{(i)}_R$ (i.e.\@ a midpoint of $I^{(i)}$) for $i \in \left\{ 1, \dots, k-1 \right\}$ and let $x^{(i)}_j \in I^{(i)}$ for each $(i,j) \in \left\{ 1, \dots, r \right\} \times \left\{ 1, \dots, n_i \right\}$, $(i,j) \notin \left\{ 1, \dots, k-1 \right\} \times \left\{ 1,2 \right\}$. Then for each $i \in \left\{ 1, \dots, k-1 \right\}$ we have
\begin{align*} |x^{(i)}_1 - x^{(i)}_2| &= |x^{(i)}_1 - x^{(i)}_0| + |x^{(i)}_0 - x^{(i)}_2|, \\
|x^{(i)}_j - x^{(i)}_1| + |x^{(i)}_j - x^{(i)}_2| &\geq |x^{(i)}_j - x^{(i)}_0| \,\, \textnormal{for each} \,\, j \in \left\{ 3, \dots, n_i \right\}.
\end{align*}
We can use this to bound the expression under the brackets on the right hand side of (\ref{CalZyg}) from below with
\begin{align*}\sum_{i=1}^r \sum_{1 \leq j_1 < j_2 \leq n_i} |x_{j_1}^{(i)}-x_{j_2}^{(i)}| &\geq \sum_{i=1}^{k-1} \sum_{j=1}^{n_i} |x^{(i)}_j-x^{(i)}_0| + \sum_{j=2}^{n_k} |x^{(k)}_j - x^{(k)}_1| \\
&\geq \bigg( \sum_{i=1}^{k-1} \sum_{j=1}^{n_i} |x^{(i)}_j-x^{(i)}_0|^2 + \sum_{j=2}^{n_k} |x^{(k)}_j - x^{(k)}_1|^2 \bigg)^{\frac{1}{2}}.\end{align*}
Let $\mathbbm{x}_0 := ( \underbrace{x^{(1)}_0, \dots, x^{(1)}_0}_{n_1 \,\, \textnormal{times}}, \dots, \underbrace{x^{(k-1)}_0, \dots, x^{(k-1)}_0}_{n_{k-1} \,\, \textnormal{times}}, \underbrace{x^{(k)}_1, \dots, x^{(k)}_1}_{n_k-1 \,\, \textnormal{times}})$. Note that 
\[( \prod_{i=1}^{k-1} \prod_{j=1}^{n_i} I^{(i)}_j ) \times ( \prod_{j=2}^{n_k} I^{(k)}_j ) \subseteq B ( \mathbbm{x}_0, n |I^{(1)}|),\]
where the latter set is a $(n_1+\dots+n_k-1)$-dimensional ball with the center $\mathbbm{x}_0$ and a radius $n|I^{(1)}_1|$. Using this, the inequality from above that we showed earlier and the integration in spherical coordinates, for all possible choices of $I^{(i)}_j \in \{ I^{(i)}_L, I^{(i)}_R \}$, where $(i,j) \notin \left\{ 1, \dots, k-1 \right\} \times \left\{ 1,2 \right\}$, we get
{\allowdisplaybreaks
\begin{align*} |Q|^{-1} &\bigg| \int_{\mathbb{R}^n} K(\mathbbm{x}) \otimes_{i=1}^r \otimes_{j=1}^{n_i} \mathbbm{1}_{I^{(i)}_j}(\mathbbm{x}) d\mathbbm{x} \bigg| \leq |Q|^{-1} \int_{\prod_{i=1}^r \prod_{j=1}^{n_i} I^{(i)}_j} |K(\mathbbm{x})| d\mathbbm{x} \\
&\leq |Q|^{-1} \int_{\prod_{i=k}^r I^{(i)}} \int_{B ( \mathbbm{x}_0, n |I^{(1)}|)} |K(\mathbbm{x})| \bigg( \prod_{i=1}^{k-1} \prod_{j=1}^{n_i} dx^{(i)}_j \cdot \prod_{j=2}^{n_k} dx^{(i)}_j \bigg) \prod_{i=k}^r dx^{(i)}_1 \\
&\leq |I^{(1)}|^{-r} \int_{\prod_{i=k}^r I^{(i)}} \int_0^{n |I^{(1)}|} t^{r-n} \cdot t^{\sum_{i=1}^k n_i - 2} dt \prod_{i=k}^r dx^{(i)}_1 \\
&\lesssim |I^{(1)}|^{-r} \cdot |I^{(1)}|^{k-1} \cdot |I^{(1)}|^{r-k+1} = 1.
\end{align*}
}
Since the choice of $Q \in \mathcal{C}_r$ was arbitrary, we conclude $\left\| \lambda \right\|_{\ell^{\infty}(\mathcal{C}_r)} \lesssim 1$.

(b) Just as we showed at the beginning of the proof of Proposition \ref{tree}, we can, without loss of generality, assume $\left| Q_{\mathcal{T}} \right| = 1$ and $\max_{Q \in \mathcal{T} \cup \mathcal{L} \left( \mathcal{T} \right)} \left[ F_e^{d_e} \right]_Q^{\frac{1}{d_e}} = 1$ for each $e \in E$. Also, notice that the result for the case (C1) already follows from Proposition \ref{tree}, also using $|\lambda_Q| \leq \left\| \lambda \right\|_{\ell^{\infty}(\mathcal{C}_r)}$ for each $Q \in \mathcal{T}$.

As for the case (C2), let $H_1$ and $H_2$ be the connected components of $H$ such that each of them has at least one selected vertex. If there are $k$ connected components altogether, we can estimate
\begin{align*}| \big[ \mathbf{F}_E \big]_{H,S, Q} | &= \prod_{l=1}^k | \big[ \mathbf{F}_{E_l} \big]_{H_l,S_l, Q} | \leq | \big[ \mathbf{F}_{E_1} \big]_{H_1,S_1, Q} | \cdot | \big[ \mathbf{F}_{E_2} \big]_{H_2,S_2, Q} | \\
&\leq \frac{1}{2} ( | \big[ \mathbf{F}_{E_1} \big]_{H_1,S_1, Q} |^2 + | \big[ \mathbf{F}_{E_2} \big]_{H_2,S_2, Q} |^2 ),\end{align*}
where we used Lemma \ref{holder} applied to the hypergraphs $H_3, \dots, H_k$. We can rewrite this inequality as
$$| \big[ \mathbf{F}_E \big]_{H,S, Q} | \leq \frac{1}{2} ( | \big[ \mathbf{F}_{E_1} \big]_{H_1',S_1', Q} |^2 + | \big[ \mathbf{F}_{E_2} \big]_{H_2',S_2', Q} |^2 ),$$
where $H_l'$ is a new hypergraph consisting of two copies of the hypergraph $H_l$ and, similarly, $S_l'$ has same vertices as $S_l$ along with its analogous copies, for $l=1,2$. Formally, we construct the hypergraph $H_l' = (V_l',E_l')$ such that, for each vertex $v^{(i)} \in V_l$ we add both $v^{(i)}$ and a new vertex $v^{'(i)}$, also keeping the agreement that, for each newly constructed vertices $v_1^{'(i)}$ and $v_2^{'(i)}$, the label $x_{v_1^{'(i)}}$ is the copy of the label $x_{v_2^{'(i)}}$ if and only if the label $x_{v_1^{(i)}}$ is the copy of the label $x_{v_2^{(i)}}$; also, no label of the newly constructed vertex is a copy of the label of any vertex from $V_l$. Analogously, we define
$$E_l' := E_l \cup \{ (v^{'(1)},\dots,v^{'(r)}) : (v^{(1)},\dots,v^{(r)}) \in E \} \,\,\textnormal{and}\,\, S_l' := S_l \cup \{ v^{'(i)} : v^{(i)} \in S_l \}.$$
Note that both of the hypergraphs $H_1'$ and $H_2'$ belong to the case (C1), therefore for each $l=1,2$ we define
$$\Lambda_{E_1', \mathcal{T}}^{S_l'} ( \mathbf{F}_{E_l'} ) := \sum_{Q \in \mathcal{T}} |Q| [ \mathbf{F}_{E_1'} ]_{H_l',S_l',Q}.$$
By Proposition \ref{tree},
$$|\Lambda_{E,\mathcal{T}}^S \left( \mathbf{F}_E \right)| \leq \frac{1}{2} \big( \Lambda_{E_1', \mathcal{T}}^{S_1'} ( \mathbf{F}_{E_1'} ) + \Lambda_{E_2', \mathcal{T}}^{S_2'} ( \mathbf{F}_{E_2'} ) \big) \lesssim 1.$$
Notice that the thresholds $d_e$ required for this result are those thresholds that we get while applying the Proposition \ref{tree} on the modified hypergraphs. However, with this construction the thresholds cannot increase and are still at most equal the quantity defined in (\ref{thresholds}). This completes the proof of the proposition. \qedhere
\end{proof}

\begin{proposition} \label{noncanc} Let $\Lambda_E^S$ be a non-cancellative entangled dyadic paraproduct.
\begin{enumerate}[label=(\alph*)]
	\item If (\ref{second}) holds, then for the corresponding coefficients $\lambda^S = (\lambda^S_Q)_{Q \in \mathcal{C}_r}$ we have
    $$\left\| \lambda^S \right\|_{\textnormal{bmo}} := \sup_{Q_0 \in \mathcal{C}_r} \bigg( \frac{1}{|Q_0|} \sum_{\substack{Q \in \mathcal{C}_r \\ Q \subseteq Q_0}} |Q| |\lambda^S_Q|^2 \bigg)^{\frac{1}{2}} \lesssim 1.$$
    \item For a finite convex tree $\mathcal{T}$ and a localized non-cancellative entangled dyadic paraproduct $\Lambda_{E,\mathcal{T}}^S$ we have
    $$|\Lambda_{E,\mathcal{T}}^S \left( \mathbf{F} \right)| \lesssim \left\| \lambda^S \right\|_{\textnormal{bmo}} |Q_{\mathcal{T}}| \prod_{e \in E} \max_{Q \in \mathcal{T} \cup \mathcal{L} \left( \mathcal{T} \right)} [ F_e^{d_e} ]_Q^{\frac{1}{d_e}}.$$
\end{enumerate}
\end{proposition}

\begin{proof} (a) Let us see what we can conclude with the assumption of (\ref{second}). Fix $e_0 = (v^{(1)}, \dots, v^{(r)}) \in E$. By the definition of the operator $T_{e_0}$ given in (\ref{pomt}), in this case with kernel defined as in (\ref{kernel}), we have
    \begin{align*}T_{e_0} \left( \mathbf{F}_{E \backslash \left\{ e_0 \right\}} \right) (\mathbbm{x}_{e_0}) =& \sum_{\substack{S = (S^{(i)})_{i=1}^r \\ \left( \forall i \in \left\{ 1, \dots, r \right\} \right) S^{(i)} \subseteq V^{(i)} \\ \left( \exists i_0 \in \left\{ 1, \dots, r \right\} \right) \left| S^{(i_0)} \right| \neq 0}} \sum_{Q = \prod_{i=1}^r \left( I^{(i)} \right)^{n_i} \in \mathcal{C}_n} |Q|^{\frac{1}{2}}\lambda^S_Q \\
    &\int_{\mathbb{R}^{n-r}} \bigg( \prod_{e \in E \backslash \left\{ e_0 \right\}} F_e (\mathbbm{x}_e) \bigg) \mathbbm{h}_Q^S(\mathbbm{x}) \prod_{v \in V \backslash e_0} dx_v.\end{align*}
    We turn our attention to the case when $F_e = \mathbbm{1}_{\mathbb{R}^r}$ for each $e \in E \backslash \{ e_0 \}$. The function appearing under the integral sign in that case is $\mathbbm{h}^S_Q$ which, up to the constant $|Q|^{\frac{1}{2}}$, equals the product of functions of one variable $\mathbbm{h}^1_{I_i}$ and $\mathbbm{h}^0_{I_i}$ for each $i \in \{ 1, \dots, r \}$, where $Q = \prod_{i=1}^r (I^{(i)})^{n_i} \in \mathcal{C}_n$. Depending on whether the cancellation appears or not, the function $T_{e_0} \left( \mathbbm{1}_{\mathbb{R}^r}, \dots, \mathbbm{1}_{\mathbb{R}^r} \right)$ can either be identically equal to zero or, if $v^{(1)}, \dots, v^{(s)}$ are all the selected vertices for $s \in \mathbb{N}$, it can be given as 
    \begin{align*}T_{e_0} \left( \mathbbm{1}_{\mathbb{R}^r}, \dots, \mathbbm{1}_{\mathbb{R}^r} \right) =& \sum_{\substack{S = (S^{(1)}, \dots, S^{(s)}, \emptyset, \dots, \emptyset) \\ \left( \forall i \in \left\{ 1, \dots, s \right\} \right) S^{(i)} \subseteq \{ v^{(i)} \}\\ \left( \exists i_0 \in \left\{ 1, \dots, s \right\} \right) \left| S^{(i_0)} \right| \neq 0}} \sum_{Q = \prod_{i=1}^r \left( I^{(i)} \right)^{n_i} \in \mathcal{C}_n} |Q| \lambda^S_Q |I^{(1)}|^{r-n} \bigotimes_{i=1}^r \mathbbm{h}_{I^{(i)}}^{v^{(i)}},\end{align*}
    where we define $\mathbbm{h}_{I^{(i)}}^{v^{(i)}}$ as $\mathbbm{h}_{I^{(i)}}^1$ if $1 \leq i \leq s$ or as $\mathbbm{h}_{I^{(i)}}^0$ otherwise. From the definition of the dyadic BMO-seminorm, taking care of the cancellation again (which happens to appear in at least one variable of each summand of the above expression), we have
    \begin{align*}\| T_{e_0} ( \mathbbm{1}_{\mathbb{R}^r}, \dots, \mathbbm{1}_{\mathbb{R}^r} &) \|_{\textnormal{BMO}(\mathbb{R}^r)} \\
    &= \sup_{Q_0 \in \mathcal{C}_r} \bigg( \frac{1}{|Q_0|} \sum_{\substack{S = (S^{(1)}, \dots, S^{(s)}, \emptyset, \dots, \emptyset) \\ \left( \forall i \in \left\{ 1, \dots, s \right\} \right) S^{(i)} \subseteq \{ v^{(i)} \}\\ \left( \exists i_0 \in \left\{ 1, \dots, s \right\} \right) \left| S^{(i_0)} \right| \neq 0}} \sum_{\substack{Q = \prod_{i=1}^r (I^{(i)})^{n_i} \in \mathcal{C}_n \\ \prod_{i=1}^r I^{(i)} \subseteq Q_0}}  |I^{(1)}|^r | \lambda^S_Q |^2 \bigg)^{\frac{1}{2}}.\end{align*}
    From this, recognizing the expression inside the brackets as the bmo-norms of the coefficients, it follows that from each such choice of $S$ we have
    $$\left\| \lambda^S \right\|_{\textnormal{bmo}} \leq \left\| T_{e_0} \left( \mathbbm{1}_{\mathbb{R}^r}, \dots, \mathbbm{1}_{\mathbb{R}^r} \right) \right\|_{\textnormal{BMO}(\mathbb{R}^r)} \lesssim 1.$$
    
    Notice that for the preceding proof we were required to have an edge $e_0 \in E$ that contains all of the selected vertices from the starting hypergraph, which is precisely the condition (NC) together with completeness of the corresponding hypergraph component.
    
	(b) Without loss of generality we can assume that
	$$\left| Q_{\mathcal{T}} \right| = 1 \,\, \textnormal{and} \,\, \max_{Q \in \mathcal{T} \cup \mathcal{L} \left( \mathcal{T} \right)} [ F_e^{d_e} ]_Q^{\frac{1}{d_e}} = 1 \,\, \textnormal{for each} \,\, e \in E.$$ By the Cauchy-Schwarz inequality we have
    $$|\Lambda_{E,\mathcal{T}}^S \left( \mathbf{F} \right)| = \sum_{Q \in \mathcal{T}} |Q| |\lambda_Q| \left[ \mathbf{F} \right]_{H,S,Q} \leq \big( \sum_{Q \in \mathcal{T}} |Q| |\lambda_Q|^2 \big)^{\frac{1}{2}} \big( \sum_{Q \in \mathcal{T}} |Q| \left[ \mathbf{F} \right]_{H,S,Q}^2 \big)^{\frac{1}{2}}.$$
    We can notice that
    $$\sum_{Q \in \mathcal{T}} |Q| |\lambda_Q|^2 \leq \sum_{\substack{Q \in \mathcal{C}_r \\ Q \subseteq Q_{\mathcal{T}}}} |Q| |\lambda_Q|^2 \leq |Q_{\mathcal{T}}| \left\| \lambda^S \right\|_{\textnormal{bmo}}^2 = \left\| \lambda^S \right\|_{\textnormal{bmo}}^2.$$
	Let $H'$ be a hypergraph consisting of two copies of the hypergraph $H$ (up to the labels of vertices and edges) and let $S'$ be an $r$-tuple consisting of the vertices from $S$ and their corresponding copies. This hypergraph belongs to the case (C1), for which we already have
    $$\sum_{Q \in \mathcal{T}} |Q| \left[ \mathbf{F} \right]_{H,S,Q}^2 = \sum_{Q \in \mathcal{T}} |Q| \left[ \mathbf{F} \right]_{H',S',Q} \lesssim 1.$$
    All together, we achieve the desired claim: $\Lambda_{E,\mathcal{T}}^S \left( \mathbf{F} \right) \lesssim \left\| \lambda^S \right\|_{\textnormal{bmo}}^2$. \qedhere
\end{proof}

\section{Proof of the T(1) theorem} \label{petaprava}

\begin{proof}[Proof of Theorem \ref{T1thm}] (a) $\Rightarrow$ (e) For each $Q_0 \in \mathcal{C}_r$ denote $\mathcal{D}(Q_0) := \left\{ Q \in \mathcal{C}_r : Q \subseteq Q_0 \right\}$ and $M :=  \frac{\log_2(2|E|)}{\min_{e \in E} d_e}$. For a fixed $e \in E$ let us define
$$\mathcal{I}_{Q_0}^e := \{ Q \in \mathcal{D}(Q_0) : [F_e^{d_e}]_Q^{\frac{1}{d_e}} > 2^M [F_e^{d_e}]_{Q_0}^{\frac{1}{d_e}} \}.$$
Then define $\mathcal{M}_{Q_0}$ to be the collection of maximal cubes in $\cup_{e \in E} \mathcal{I}_{Q_0}^e$ and finally set $\mathcal{M}_{Q_0}^e := \mathcal{M}_{Q_0} \cap \mathcal{I}_{Q_0}^e$. Consequently, $\mathcal{M}_{Q_0} = \cup_{e \in E} \mathcal{M}_{Q_0}^e$, but the union does not have to be disjoint. From these definitions we have
$$\sum_{Q \in \mathcal{M}_{Q_0}^e}\!\!\! |Q| \leq \!\!\!\sum_{Q \in \mathcal{M}_{Q_0}^e}\!\!\! 2^{-Md_e} \big[ F_e^{d_e} \big]_{Q_0}^{-1} \int_Q F_e(\mathbbm{x}_e)^{d_e}d\mathbbm{x}_e \leq (2|E|)^{-1}\big[ F_e^{d_e} \big]_{Q_0}^{-1} \int_{Q_0} F_e(\mathbbm{x}_e)^{d_e}d\mathbbm{x}_e = \frac{|Q_0|}{2|E|}.$$
In the second inequality we used the fact that the elements of $\mathcal{M}_{Q_0}^{e}$ are mutually disjoint (by maximality), allowing us to increase the sum to the integral over the largest cube $Q_0$. This gives us
\begin{equation}\sum_{Q \in \mathcal{M}_{Q_0}} |Q| \leq \sum_{e \in E} \sum_{Q \in \mathcal{M}_{Q_0}^e} |Q| \leq \frac{|Q_0|}{2}. \label{eq:halfthemeasure}\end{equation}
Now, choose $Q_1, \dots, Q_{2^r} \in \mathcal{C}_r$ such that $\cup_{i=1}^{2^r} Q_i \supset \cup_{e \in E} \, \textnormal{supp} \, F_e$. Indeed, if the supports of functions $F_e$ are contained in more than one quadrant of the space $\mathbb{R}^r$, we may need at most $2^r$ dyadic cubes that cover their supports. For each $i \in \{ 1, \dots, 2^r \}$ we inductively define
\begin{align*}
 \mathcal{S}_{\mathcal{D},i,0} := \{ Q_i \},\quad \mathcal{S}_{\mathcal{D},i,n} := \cup_{Q \in \mathcal{S}_{\mathcal{D},i,n-1}} \mathcal{M}_Q, n \in \mathbb{N},\\
  \mathcal{S}_{\mathcal{D},i} := \cup_{n=0}^{\infty} \mathcal{S}_{\mathcal{D},i,n}, i \in \{ 1, \dots, 2^r \},\quad \mathcal{S}_{\mathcal{D}} := \cup_{i=1}^{2^r} \mathcal{S}_{\mathcal{D},i}.
\end{align*}
Let us notice that $\mathcal{S}_{\mathcal{D}}$ is a sparse family of dyadic cubes. Indeed, for any $Q \in \mathcal{S}_{\mathcal{D}}$ let $E_Q := Q \backslash \big( \cup_{Q' \in \mathcal{M}_Q} Q' \big).$ For each two dyadic cubes $Q_1, Q_2 \in \mathcal{S}_{\mathcal{D}},$ $Q_1 \neq Q_2$, we have that they are either mutually disjoint, therefore $E_{Q_1}$ and $E_{Q_2}$ are mutually disjoint as well, or, without loss of generality, $Q_2 \subseteq Q_1$, in which case, by construction, $Q_2 \subseteq Q'_1 \in \mathcal{I}_{Q_1}$, so $Q_2 \cap E_{Q_1} = \emptyset$, therefore $E_{Q_1}$ and $E_{Q_2}$ are again mutually disjoint. Also, for each $Q \in \mathcal{S}_{\mathcal{D}}$, by \eqref{eq:halfthemeasure} we have
$$|E_Q| = |Q| - \sum_{Q' \in \mathcal{I}_Q} |Q'| \geq \frac{1}{2} |Q|.$$
Now, for each $Q \in \mathcal{S}_{\mathcal{D}}$ and a fixed $N \in \mathbb{N}$ let us define
$$\mathcal{T}_Q^N := \mathcal{C}^N \cap \mathcal{D}(Q) \backslash \big( \cup_{Q' \in \mathcal{M}_{Q}} \mathcal{D}(Q') \big),$$
where
$$\mathcal{C}^N := \bigg\{ \prod_{i=1}^r I_i \in \mathcal{C}_r : | I_1 | = \dots = | I_r | \geq 2^{-N} \bigg\}.$$
Notice that $\mathcal{T}_Q^N$ is a finite convex tree where the set of leaves $\mathcal{L}(\mathcal{T}_Q^N)$ are either elements of $\mathcal{M}_Q$ or they are dyadic cubes with length of each side equal to $2^{-N-1}$. An application of Propositions \ref{canc} and \ref{noncanc} gives us
$$\big|\Lambda_{E,\mathcal{T}_Q^N}^S \left( \mathbf{F} \right)\big| \lesssim |Q| \prod_{e \in E} \max_{Q' \in \mathcal{T}_Q^N \cup \mathcal{L} \left( \mathcal{T}_Q^N \right)} [ F_e^{d_e} ]_{Q'}^{\frac{1}{d_e}}.$$
If $Q' \in \mathcal{T}_Q^N$ then $Q' \notin \mathcal{M}_Q^e$, which means that $[ F_e^{d_e} ]_{Q'}^{\frac{1}{d_e}} \leq 2^M [ F_e^{d_e} ]_Q^{\frac{1}{d_e}}$. If $Q' \in \mathcal{L} ( \mathcal{T}_Q^N ) \cap \mathcal{S}_{\mathcal{D}}$ and $Q'_P$ is a parent of $Q'$, then by maximality we have
$$[ F_e^{d_e} ]_{Q'}^{\frac{1}{d_e}} \leq 2^{\frac{r}{d_e}} [ F_e^{d_e} ]_{Q'_P}^{\frac{1}{d_e}} \leq 2^{r+M} [ F_e^{d_e} ]_Q^{\frac{1}{d_e}}.$$
The remaining option is if each side of $Q'$ has length equal to $2^{-N-1}$. But even then its parent $Q'_P$ satisfies $Q'_P \notin \mathcal{M}_Q^e$, so that the above inequality is valid again. Altogether,
$$|\Lambda_{E,\mathcal{T}_Q^N}^S| \left( \mathbf{F} \right) \lesssim 2^{|E|M} |Q| \prod_{e \in E} [ F_e^{d_e} ]_Q^{\frac{1}{d_e}}.$$
This holds for any $Q \in \mathcal{S}_{\mathcal{D}}$. Note that the trees $\mathcal{T}_Q^N$, $Q \in \mathcal{S}_{\mathcal{D}}$ form a partition of $\left( \cup_{i=1}^{2^r} \mathcal{D}(Q_i) \right) \cap \mathcal{C}^N$, therefore
$$|\Lambda_{E,\left( \cup_{i=1}^{2^r} \mathcal{D}(Q_i) \right) \cap \mathcal{C}^N}^S \left( \mathbf{F} \right)| \lesssim 2^{|E|M}  \sum_{Q \in \mathcal{S}_{\mathcal{D}}} |Q| \prod_{e \in E} [ F_e^{d_e} ]_Q^{\frac{1}{d_e}}.$$
As the right side of the inequality and the inequality itself does not depend on $N$, with $N \rightarrow \infty$ we get
$$|\Lambda_{E,\cup_{i=1}^{2^r} \mathcal{D}(Q_i)}^S \left( \mathbf{F} \right)| \lesssim 2^{|E|M}  \sum_{Q \in \mathcal{S}_{\mathcal{D}}} |Q| \prod_{e \in E} [ F_e^{d_e} ]_Q^{\frac{1}{d_e}}.$$
Note that for $Q \in \mathcal{C}_r, Q \notin \cup_{i=1}^{2^r} \mathcal{D}(Q_i)$ we have that the summand in (\ref{localtree}) equals zero, so this inequality can be rewritten as
$$| \Lambda_E( \mathbf{F} ) | \lesssim \Theta_{\mathcal{S}_{\mathcal{D}}}( \mathbf{F} ),$$
with a sparse form given associated with the sparse family $\mathcal{S}_{\mathcal{D}}$. 

(e) $\Rightarrow$ (f) Let $\Theta_{\mathcal{S}}$ be the sparse form that bounds the form $\Lambda_E$. It will be enough to prove the analogous inequality for $\Theta_{\mathcal{S}}$. Once again, it is sufficient to work with nonnegative functions $F_e$. For each $e \in E$ let $h_e := w_e^{\frac{-d_e}{p_e-d_e}}$ and let $G_e$ be a function such that $F_e = G_e h_e^{\frac{1}{d_e}}$. Note that we have $\| F_e \|_{\textnormal{L}^{p_e}(w_e)} = \| G_e \|_{\textnormal{L}^{p_e}(h_e)}$. Let us rewrite the form $\Theta_{\mathcal{S}}$ in the following way:
\begin{align}\Theta_{\mathcal{S}}( \mathbf{F} ) = \sum_{Q \in \mathcal{S}} &\bigg( \prod_{e \in E} [ h_e ]_Q^{\frac{1}{d_e}-\frac{1}{p_e}} \bigg) \bigg( |Q| \prod_{e \in E} \bigg( \frac{[h_e]_Q}{|E_Q|[h_e]_{E_Q}} \bigg)^{\frac{1}{p_e}} \bigg) \nonumber\\
&\bigg( \prod_{e \in E} (|E_Q|[h_e]_{E_Q})^{\frac{1}{p_e}} \bigg( \frac{[G_e^{d_e}h_e]_Q}{[h_e]_Q} \bigg)^{\frac{1}{d_e}} \bigg). \label{eq:weightedsplitting}\end{align}
We can see directly from the definition of the Muckenhoup constant that $\prod_{e \in E} [ h_e ]_Q^{\frac{1}{d_e}-\frac{1}{p_e}} \leq [\mathbf{w}]_{\mathbf{p},\mathbf{d}}$. To bound the expression inside the second pair of parentheses, first notice that, by the H\"older inequality, by (\ref{condMuck}) and along with with $r_e := \frac{p_e-d_e}{p_ed_e}$ for each $e \in E$, $r := \sum_{e \in E} r_e$ and the constant $c$ from Definition \ref{defsparse} for the given family $\mathcal{S}$ we can see that
$$\prod_{e \in E} (|E_Q|[h_e]_{E_Q})^{\frac{r_e}{r}} = \prod_{e \in E} \bigg( \int_{E_Q} h_e(x)dx \bigg)^{\frac{r_e}{r}} \geq \int_{E_Q} \prod_{e \in E} h_e(x)^{\frac{r_e}{r}}dx = |E_Q| \geq c|Q|.$$ 
for each $Q \in \mathcal{S}$. Denote $m := \max_{e \in E} \frac{1}{r_ep_e}$. This gives us
$$\prod_{e \in E} \bigg( \frac{|Q|[h_e]_Q}{|E_Q|[h_e]_{E_Q}} \bigg)^{\frac{1}{p_e}} \leq \prod_{e \in E} \bigg( \frac{|Q|}{|E_Q|[h_e]_{E_Q}} \bigg)^{r_e m} [h_e]_Q^{r_e m} \leq c^{-rm} [\mathbf{w}]_{\mathbf{p},\mathbf{d}}^m.$$
It remains to note that we have already obtained one power of the Muckenhoupt constant and observe that
\[ 1 + m = 1 + \max_{e\in E} \frac{d_e}{p_e-d_e} = \max_{e\in E} \frac{p_e}{p_e-d_e}. \]
We have
$$\bigg( \prod_{e \in E} [ h_e ]_Q^{\frac{1}{d_e}-\frac{1}{p_e}} \bigg) |Q| \prod_{e \in E} \bigg( \frac{[h_e]_Q}{|E_Q|[h_e]_{E_Q}} \bigg)^{\frac{1}{p_e}} \lesssim [\mathbf{w}]_{\mathbf{p},\mathbf{d}}^{\max_{e \in E} \frac{p_e}{p_e-d_e}},$$
with the implicit constant depending only on $c$. Note that the expressions in the first two parentheses of \eqref{eq:weightedsplitting} are bounded uniformly in $Q \in \mathcal{S}$. Now let us define the weighted maximal operator with
$$M_{d,w}F(x_1,\dots,x_r) := \sup_{\substack{Q \in \mathcal{C}_r \\ (x_1,\dots,x_r) \in Q}} \left( \frac{[|F|^d w]_Q}{[w]_Q} \right)^{\frac{1}{d}}.$$
This type of operator is bounded on the weighted space $\textup{L}^p(w)$ for each $p>d$, which is a result from \cite{M12}. Since the first  two terms in \eqref{eq:weightedsplitting} are bounded independently of $Q$, we turn to the sum of the third terms over $Q \in \mathcal{S}$:
\begin{align*}\sum_{Q \in \mathcal{S}} \prod_{e \in E} \bigg( \int_{E_Q} h_e(x) dx \bigg)^{\frac{1}{p_e}} \bigg( \frac{[ G_e^{d_e}h_e ]_Q}{[h_e]_Q} \bigg)^{\frac{1}{d_e}} &= \sum_{Q \in \mathcal{S}} \prod_{e \in E} \bigg( \int_{E_Q} \bigg( \frac{[G_e^{d_e}h_e]_Q}{[h_e]_Q} \bigg)^{\frac{p_e}{d_e}} h_e(x)dx \bigg)^{\frac{1}{p_e}} \\
&\leq \sum_{Q \in \mathcal{S}} \prod_{e \in E} \bigg( \int_{E_Q} (M_{d_e,h_e}G_e)(x)^{p_e} h_e(x)dx \bigg)^{\frac{1}{p_e}}.
\end{align*}
By H\"{o}lder's inequality for the summation in $Q$, the disjointness of $E_Q$ and boundedness of $M_{d_e,h_e}$ the last expression is at most
\begin{align*}\prod_{e \in E} \bigg( \sum_{Q \in \mathcal{S}} \int_{E_Q} (M_{d_e,h_e}G_e)(x)^{p_e} h_e(x)dx \bigg)^{\frac{1}{p_e}} &\leq \prod_{e \in E} \| M_{d_e,h_e}G_e \|_{\textnormal{L}^{p_e}(h_e)} \lesssim \prod_{e \in E} \| G_e \|_{\textnormal{L}^{p_e}(h_e)} \\
&= \prod_{e \in E} \| F_e \|_{\textnormal{L}^{p_e}(w_e)}
\end{align*}
which gives the desired weighted estimate.

(f) $\Rightarrow$ (c) The required bound follows if we use $\mathbf{w} = (w_e)_{e \in E}$ given with $w_e := \mathbbm{1}_{\mathbbm{R}^r}$ for each $e \in E$.

(c) $\Rightarrow$ (d) This implication is trivial.

(d) $\Rightarrow$ (b) Let $p_e \in \left< d_e, \infty \right], e \in E$ be such that (\ref{T1bound}) is valid and let $e_0 \in E$ and $Q \in \mathcal{C}_r$ be arbitrary. Specially, if we take $F_e = \mathbbm{1}_Q$ for each $e \in E \backslash \{ e_0 \}$, we have
\begin{align*}\bigg| \int_{\mathbb{R}^r} T_{e_0}(\mathbbm{1}_Q)_{e \in E \backslash \{ e_0 \}} F_{e_0}(\mathbbm{x}_{e_0})(\mathbbm{x}_{e_0}) d\mathbbm{x}_{e_0} \bigg| &= | \Lambda_E (\mathbf{F}) | \lesssim \| F_{e_0} \|_{\textnormal{L}^{p_{e_0}}(\mathbb{R}^r)} \prod_{e \in E \backslash \{ e_0 \}} \| \mathbbm{1}_Q \|_{\textnormal{L}^{p_e}(\mathbb{R}^r)} \\
&= \| F_{e_0} \|_{\textnormal{L}^{p_{e_0}}(\mathbb{R}^r)} |Q|^{\sum_{e \in E \backslash \{ e_0 \}}\frac{1}{p_e}} = \| F_{e_0} \|_{\textnormal{L}^{p_{e_0}}(\mathbb{R}^r)} |Q|^{\frac{1}{q_{e_0}}},
\end{align*}
where $q_{e_0}$ is the conjugated exponent of $p_{e_0}$. This gives us
$$\| T_{e_0}(\mathbbm{1}_Q)_{e \in E \backslash \{ e_0 \}} \|_{\textnormal{L}^{q_{e_0}}(Q)} \lesssim |Q|^{\frac{1}{q_{e_0}}}.$$
Combining this with Jensen's inequality,
$$\frac{1}{|Q|} \int_Q | T_{e_0}(\mathbbm{1}_Q)_{e \in E \backslash \{ e_0 \}} (\mathbbm{x}_{e_0}) | d\mathbbm{x}_{e_0} \leq \bigg( \frac{1}{|Q|} \int_Q | T_{e_0}(\mathbbm{1}_Q)_{e \in E \backslash \{ e_0 \}} (\mathbbm{x}_{e_0}) |^{q_{e_0}} d\mathbbm{x}_{e_0} \bigg)^{\frac{1}{q_{e_0}}} \lesssim 1,$$
which shows that the condition (\ref{assum}) is valid. 

(b) $\Rightarrow$ (a) Note that from the inequality (\ref{assum}) for any $Q \in \mathcal{C}_r$ we have
\begin{align*}|\Lambda_E((\mathbbm{1}_Q)_{e \in E} )| &= \bigg| \int_{\mathbb{R}^r} T_{e_0} \left( (\mathbbm{1}_Q)_{e \in E \backslash \left\{ e_0 \right\}} \right) (\mathbbm{x}_{e_0}) \mathbbm{1}_Q(\mathbbm{x}_{e_0}) d\mathbbm{x}_{e_0} \bigg| \\
&\leq \| T_{e_0} \left( (\mathbbm{1}_Q)_{e \in E \backslash \left\{ e_0 \right\}} \right) \|_{\textnormal{L}^1(Q)} \lesssim |Q|.\end{align*}
This shows us (\ref{first}) from the statement of Theorem \ref{T1thm}. Take $r > 0$ such that the support of the kernel $K$ is contained in $\left[ -r,r \right]^n$. Let $e_0 \in E$ and $Q_{e_0} \in \mathcal{C}_r$ be arbitrary. Define
$$\mathcal{S}(Q_{e_0}) := \{ Q' \in \mathcal{C}_r : |Q'| = |Q_{e_0}| \,\, \textnormal{and} \,\, | Q' \cap \left[ -r,r \right]^n | > 0 \}.$$
Note that 
\begin{align*}T_{e_0} (( \mathbbm{1}_{\mathbb{R}^r} )_{e \in E \backslash \{ e_0 \}} )(\mathbbm{x}_{e_0}) \mathbbm{1}_{Q_{e_0}}(\mathbbm{x}_{e_0}) &= \sum_{e \in E \backslash \{ e_0 \}} \sum_{Q_e \in \mathcal{S}(Q)} T_{e_0} (( \mathbbm{1}_{Q_{e'}} )_{e' \in E \backslash \{ e_0 \}} )(\mathbbm{x}_{e_0}) \mathbbm{1}_{Q_{e_0}}(\mathbbm{x}_{e_0}) \\
&= \sum_{e \in E \backslash \{ e_0 \}} \sum_{Q_e \in \mathcal{S}(Q)} \int_{\mathbb{R}^{n-r}} \bigg( \prod_{e' \in E} \mathbbm{1}_{Q_{e'}} (\mathbbm{x}_{e'}) \bigg) K(\mathbbm{x}) \prod_{v \in V \backslash e_0} dx_v.
\end{align*}
As the cubes $Q_{e'}, e' \in E$ all have equal Lebesgue measure, they are either identical or disjoint, which means that each integral expression is of the form
$$\int_{\mathbb{R}^{n-r}} \bigg( \prod_{i=1}^r \prod_{v^{(i)} \in V^{(i)}} \mathbbm{1}_{I_{v^{(i)}}} (x_{v^{(i)}}) \bigg) K(\mathbbm{x}) \prod_{v \in V \backslash e_0} dx_v,$$
for dyadic intervals $I_{v^{(i)}}, v^{(i)} \in V^{(i)}, i = 1, \dots, r$ such that $\prod_{i=1}^r \prod_{v^{(i)} \in V^{(i)}} I_{v^{(i)}} = Q_{e_0}$. As $K$ is constant on dyadic cubes $\prod_{i=1}^r \prod_{v^{(i)} \in V^{(i)}} I_{v^{(i)}}$ for which $I_{v^{(i_1)}} \neq I_{v^{(i_2)}}$ for some $v^{(i_1)}, v^{(i_2)} \in V^{(i)}$ and some $i \in \{ 1, \dots, r \}$, i.e.\@ on those cubes that do not intersect the diagonal, the above expression is the constant that coincides with its average over the same cube (the integral over the same cube divided by its Lebesgue measure). In case that for certain dyadic intervals $I_1, \dots, I_r$, we have $I_{v^{(i_1)}} = I_{v^{(i_2)}} = I_i$ for every $v^{(i_1)}, v^{(i_2)} \in V^{(i)}$ and $i \in \{ 1, \dots, r \}$, we can realize that $Q_{e_0} = I_1^{n_1} \times \dots \times I_r^{n_r} = Q_e$ for each $e \in E$, therefore the above expression takes the form
\begin{align*}T_{e_0} (( \mathbbm{1}_{Q_{e'}} )_{e' \in E \backslash \{ e_0 \}} )(\mathbbm{x}_{e_0}) \mathbbm{1}_{Q_{e_0}}(\mathbbm{x}_{e_0}) &= \int_{\mathbb{R}^{n-r}} \bigg( \prod_{i=1}^r \prod_{v^{(i)} \in V^{(i)}} \mathbbm{1}_{I_i} (x_{v^{(i)}}) \bigg) K(\mathbbm{x}) \prod_{v \in V \backslash e_0} dx_v \\
&= \int_{\mathbb{R}^{n-r}} \bigg( \prod_{e' \in E} \mathbbm{1}_{Q_{e_0}} (\mathbbm{x}_{e'}) \bigg) K(\mathbbm{x}) \prod_{v \in V \backslash e_0} dx_v \\
&= T_{e_0} (( \mathbbm{1}_{Q_{e_0}} )_{e' \in E \backslash \{ e_0 \}} )(\mathbbm{x}_{e_0}) \mathbbm{1}_{Q_{e_0}}(\mathbbm{x}_{e_0}).\end{align*}
Combining both cases, we get, for each $\mathbbm{x}_{e_0} \in Q_{e_0}$,
\begin{align*} T_{e_0} (( \mathbbm{1}_{\mathbb{R}^r} &)_{e \in E \backslash \{ e_0 \}} )(\mathbbm{x}_{e_0})- \frac{1}{|Q_{e_0}|} \int_{Q_{e_0}} T_{e_0} (( \mathbbm{1}_{\mathbb{R}^r} )_{e \in E \backslash \{ e_0 \}} )(\mathbbm{y}_{e_0})d\mathbbm{y}_{e_0} \\
&= T_{e_0} (( \mathbbm{1}_{Q_{e_0}} )_{e \in E \backslash \{ e_0 \}} )(\mathbbm{x}_{e_0}) - \frac{1}{|Q_{e_0}|} \int_{Q_{e_0}} T_{e_0} (( \mathbbm{1}_{Q_{e_0}} )_{e \in E \backslash \{ e_0 \}} )(\mathbbm{y}_{e_0})d\mathbbm{y}_{e_0}.
\end{align*}
This gives us
\begin{align*}\frac{1}{|Q_{e_0}|}\int_{Q_{e_0}} &\bigg| T_{e_0} (( \mathbbm{1}_{\mathbb{R}^r} )_{e \in E \backslash \{ e_0 \}} )(\mathbbm{x}_{e_0})- \frac{1}{|Q_{e_0}|} \int_{Q_{e_0}} T_{e_0} (( \mathbbm{1}_{\mathbb{R}^r} )_{e \in E \backslash \{ e_0 \}} )(\mathbbm{y}_{e_0})d\mathbbm{y}_{e_0} \bigg| d\mathbbm{x}_{e_0} \\
&\leq \frac{2}{|Q_{e_0}|} \int_{Q_{e_0}} | T_{e_0} (( \mathbbm{1}_{Q_{e_0}} )_{e \in E \backslash \{ e_0 \}} )(\mathbbm{x}_{e_0}) |d\mathbbm{x}_{e_0} \lesssim 1,
\end{align*}
where we applied (\ref{assum}). By the dyadic John-Nirenberg inequality which can be found in \cite{AHMTT02}, the expression
$$\sup_{Q_{e_0} \in \mathcal{C}_r} \frac{1}{|Q_{e_0}|}\int_{Q_{e_0}} \bigg| T_{e_0} (( \mathbbm{1}_{\mathbb{R}^r} )_{e \in E \backslash \{ e_0 \}} )(\mathbbm{x}_{e_0})- \frac{1}{|Q_{e_0}|} \int_{Q_{e_0}} T_{e_0} (( \mathbbm{1}_{\mathbb{R}^r} )_{e \in E \backslash \{ e_0 \}} )(\mathbbm{y}_{e_0})d\mathbbm{y}_{e_0} \bigg| d\mathbbm{x}_{e_0}$$
is comparable with $ \| T_{e_0} (( \mathbbm{1}_{\mathbb{R}^r} )_{e \in E \backslash \{ e_0 \}} ) \|_{\textnormal{BMO}(\mathbb{R}^r)}$. This shows us that (\ref{second}) is valid as well.
\end{proof}

\section*{Acknowledgements}

This work was supported in part by the Croatian Science Foundation under the project UIP-2017-05-4129 (MUNHANAP). The author also acknowledges partial support by the DAAD--MZO bilateral grant \emph{Multilinear singular integrals and applications}. The author would like to thank his advisor Vjekoslav Kova\v{c} for introducing him to the problem, for the constant support and for the maximal patience during the work on this paper. The author would also like to thank Polona Durcik for helpful comments on a preliminary draft of this paper as well as Christoph Thiele for useful suggestions and discussions.


\end{document}